\documentclass{amsart}

\usepackage[english]{babel}
\usepackage[utf8]{inputenc}
\usepackage{hyperref}
\usepackage{graphicx}
\usepackage{color}
\usepackage{caption}
\usepackage{subcaption}

\usepackage{amsmath}
\usepackage{amssymb}
\usepackage{algorithm}
\usepackage{algpseudocode}

\usepackage{enumitem}
\usepackage{diagbox}

\usepackage{mathabx}
\usepackage{tikz}
\usetikzlibrary{graphs}
\usepackage{tkz-graph}
\usetikzlibrary{fit}
\usetikzlibrary{fit}

\newtheorem{theorem}{Theorem}[section]
\newtheorem{lemma}[theorem]{Lemma}

\newtheorem{proposition}[theorem]{Proposition}

\theoremstyle{definition}
\newtheorem{definition}[theorem]{Definition}

\theoremstyle{remark}
\newtheorem{remark}[theorem]{Remark}

\newtheorem{example}{Example}

\newcommand{\abs}[1]{\lvert #1\rvert}

\calclayout

\DeclareMathOperator*{\argmin}{argmin \,}

\begin{document}

\title{
On Algorithms for and Computing with the Tensor Ring Decomposition}

\author{Oscar Mickelin}
\address{Department of Mathematics, Massachusetts Institute of Technology, Massachusetts, USA}
\email{oscarmi@mit.edu}

\author{Sertac Karaman}
\address{Department of Aeronautics and Astronautics, Massachusetts Institute of Technology, Massachusetts, USA}
\email{sertac@mit.edu}
\thanks{This work was supported in part by the National Science Foundation grant no 1350685, the Army Research Office grant no W911NF1510249 and the ARL DCIST program.}
%

%
%


\begin{abstract}
Tensor decompositions such as the canonical format and the tensor train format have been widely utilized to reduce storage costs and operational complexities for high-dimensional data, achieving linear scaling with the input dimension instead of exponential scaling. In this paper, we investigate even lower storage-cost representations in the tensor ring format, which is an extension of the tensor train format with variable end-ranks. Firstly, we introduce two algorithms for converting a tensor in full format to tensor ring format with low storage cost. Secondly, we detail a rounding operation for tensor rings and show how this requires new definitions of common linear algebra operations in the format to obtain storage-cost savings. Lastly, we introduce algorithms for transforming the graph structure of graph-based tensor formats, with orders of magnitude lower complexity than existing literature. The efficiency of all algorithms is demonstrated on a number of numerical examples, and in certain cases, we demonstrate significantly higher compression ratios when compared to previous approaches to using the tensor ring format.
\end{abstract}

\maketitle


\section{Introduction}
Tensor decompositions were originally introduced in 1927 \cite{hitchcock1927expression} and have been widely used in several fields, ranging from scientific computing to data analysis, with an initial application in psychometrics \cite{tucker1966some}. With the advent of large-scale computing over the last decades, these decompositions have become even more relevant as they circumvent the ``curse of dimensionality'' by achieving operational complexities scaling linearly with the input dimension instead of exponentially. Recent applications include machine learning \cite{novikov2015tensorizing,wang2018wide}, tensor completion \cite{yuan2018high}, and high-dimensional numerical analysis \cite{oseledets2011tensor,khoromskij2011dlog,dolgov2012fast,khoromskaia2016fast,dolgov2015simultaneous,hackbusch2012use,oseledets2010cross,grasedyck018distributed,bousse2018linear}. A large number of additional examples can be found in a recent review article \cite{kolda2009tensor} as well as a recent monograph \cite{cichocki2016tensor,cichocki2017tensor}.

We consider the tensor ring format (or TR-format) of a tensor $T\in \mathbb{R}^{n_1 \times \cdots \times n_d}$, also known as the tensor chain format in earlier mathematics literature \cite{khoromskij2011dlog}, or matrix product states format with periodic boundaries in the physics literature\cite{affleck1988valence,perez2006matrix,orus2014practical,schollwock2011density}
The format is given by tensors of the form
\begin{equation}\label{eq:TRdecomp}
T(i_1, \ldots , i_d) = \text{Trace}\left( G_1(i_1) \cdot \ldots \cdot G_d(i_d) \right),
\end{equation}
or in index-form
\begin{equation}
T(i_1, \ldots , i_d) = \sum_{\alpha_0 = 1 }^{r_{0}} \cdots  \sum_{\alpha_{d-1} = 1 }^{r_{d-1}} G_1(\alpha_0, i_1, \alpha_1) \cdot \ldots \cdot G_d(\alpha_{d-1}, i_d, \alpha_0).
\end{equation}
Here, the matrices $G_k(i_k)$ are of size $r_{k-1}\times r_k$ for each index $i_k$ with  $1 \leq i_k \leq n_k$. Each $G_k$ can therefore be viewed as an order-three tensor in $\mathbb{R}^{r_{k-1}\times n_k \times r_k}$ and is called a core tensor of the representation in Eq.~\eqref{eq:TRdecomp}. The vector $(r_0, r_1, \ldots , r_d)$ with $r_d = r_0$ is called the TR-rank of the representation in Eq.~\eqref{eq:TRdecomp}.

The TR-format can be seen as a natural extension of the successful tensor train format  (TT-format) \cite{oseledets2011tensor} where it is insisted that $r_0 = r_d = 1$, but the TR-format is known to have theoretical drawbacks in comparison. Generally, graph-based tensor formats with cycles in the associated graph are known to not be closed (in the Zariski topology) \cite{khoromskij2012tensors,landsberg2011geometry,ye2018tensor}, which may lead to concerns about numerical stability issues in analogy to a classical setting \cite{de2008tensor}. Moreover, it was recently shown that for the TR-format, minimal TR-ranks for a given tensor need not be unique \cite{ye2018tensor} (not even up to permutation of the indices $i_1, \ldots , i_d$), leading to difficulties in their calculation. On the other hand, from a pragmatically practical viewpoint, the use of this format in numerical experiments has been seen to lead to lower ranks of the core tensors as compared to the TT-format \cite{khoo2017efficient,wang2017efficient,wang2018wide,zhao2017learning,zhao2016tensor}, resulting in higher compression ratios and lower storage costs required to represent a given tensor. The mathematics literature has therefore seen renewed interest in the TR-format in recent years, aiming to build on the success of the theory and applications of the TT-format while at the same time improving the compression ratios of the involved tensors even further \cite{espig2012note,handschuh2012changing,khoo2017efficient,wang2017efficient,wang2018wide,zhao2017learning,zhao2016tensor}. Recent applications making use of the TR-format include compression of convolutional neural networks \cite{wang2018wide}, image and video compression \cite{zhao2017learning,zhao2016tensor}, tensor completion\cite{yuan2018rank} as well as image and video reconstruction \cite{wang2017efficient}. In this paper, we are interested in analysis and algorithms for computing with the TR-format that lead to higher compression ratios for representing tensors.

\subsection{Prior work}
Previous work \cite{espig2012note,zhao2017learning,zhao2016tensor} has devised efficient SVD- and alternating least-squares-based (ALS-based) approximation schemes to convert a tensor in full format to TR-format, has detailed how common linear algebra operations on the tensor level can be captured on the level of the TR-representation and has presented ways of converting a tensor into the TR-format from other common tensor formats such as the canonical or Tucker-format (see e.g., the review article by Kolda and Bader \cite{kolda2009tensor} for an overview of these tensor formats).

For ALS-based algorithms, a recent result \cite{2019arXiv190507101C} characterizes the existence and non-existence of spurious local minima in terms of the largest TR-rank. The geometric and algebraic properties of the TR-format have also been studied \cite{landsberg2011geometry,ye2018tensor}. The closure of the set of tensors with bounded TR-rank has been related to geometric complexity theory \cite{landsberg2011geometry}. The dimensions and generic ranks of the set of tensors with bounded ranks have been investigated by Ye and Lim \cite{ye2018tensor}, together with additional properties about minimality of TR-ranks and intersections of sets of tensors with differently bounded ranks.

Handschuh \cite{handschuh2012changing} provides a method of transforming a given TT-representation into TR-format with $r_0 \neq 1$. Wang et al. \cite{wang2017efficient} consider completion of a tensor in the TR-format with missing entries, using iterations of ALS. Khoo and Ying \cite{khoo2017efficient} consider the problem of converting a tensor in full format into TR-format using only a limited sample of the tensor elements $T(i_1, \ldots , i_d)$, which enables the algorithms to be applied to higher dimensional tensors. The non-sampled entries are recovered using an ALS-based iterative procedure.

\subsection{Remaining challenges}
Compared to the analogous situation for the TT-format, a number of important questions remain a challenge:
\begin{itemize}
\item \textbf{Choice of TR-rank:} For the TT-format, it is known that each $r_k$ is greater than or equal to the rank of the $k$:th unfolding matrix of the tensor $T$ (defined in Sec.~\ref{sec:notation} below) and equality is achieved when using the TT-SVD algorithm based on successive SVD-decompositions of the unfolding matrices \cite{oseledets2011tensor}. In contrast, for the TR-format, there is not a single unique minimal rank of a given tensor \cite{ye2018tensor}. The previously mentioned algorithms therefore require a manual choice of either $r_0$ (SVD-based algorithms) or the entire vector $(r_0, r_1, \ldots , r_d)$ (ALS-based algorithms). How are these ranks to be chosen, and how does the end result depend on this choice?
\item \textbf{On efficient rounding:} Arguably the most important algorithm for the TT-format is an efficient, SVD-based and non-iterative rounding procedure to convert a TT-representation with suboptimal ranks into one with more beneficial ranks. Is there an efficient analogue also in the TR-format?
\end{itemize}

\subsection{Contributions}
In order to answer these questions, we present the three main contributions in this paper.

\begin{enumerate}
\item \textbf{Importance of choice of TR-rank and heuristic algorithm:} We show that the compression ratio of the TR-format is highly dependent (in our examples, up to more than one order of magnitude) on the choice of the rank $r_0$ in Eq.~\eqref{eq:TRdecomp}. In earlier work \cite{espig2012note,zhao2017learning,zhao2016tensor}, this choice was always kept fixed and not adapted to the underlying tensor, but finding the optimal $r_0$ is crucial for the efficiency of the TR-format.  Using the notion of minimal TR-ranks defined below, we clarify why this is the case. By leveraging a certain invariance of the TR-format under cyclic shifts of the tensor dimensions $n_1, \ldots, n_d$, we show how to improve compression ratios even further. We detail a heuristic algorithm to find a low-cost choice of $r_0$ and cyclic shift in the previous SVD-based algorithms.

\item \textbf{Rounding, operations, and compressed tensor format conversions:} We describe an efficient SVD-based rounding procedure to decrease the ranks of a given TR-representation. Surprisingly, we show that earlier definitions of common linear algebra operations in the literature need to be redefined in order for ranks to be reduced when combined with the rounding procedure. As an application, this enables us to specify conversions from the TT- and canonical formats into TR-format that are more efficient than previous results in the literature \cite{zhao2017learning,zhao2016tensor}, leading to orders of magnitude higher compression ratios. Moreover, the complexity of these conversions is asymptotically lower than previous ideas \cite{handschuh2012changing} by a factor of more than $\max_k n_k^3$.

\item \textbf{Extension to graph-based formats:} We extend these ideas to general graph-based tensor formats. As an application, we present an algorithm for selecting a low-cost graphical format with which to represent a given tensor. Operations on graph-based formats also enable the aforementioned conversions from the TT- and canonical formats into TR-format.
\end{enumerate}

The remainder of the paper is organized as follows. Sec.~\ref{sec:heur} details a heuristic algorithm for choosing $r_0$ and cyclic shift to achieve low storage cost. Sec.~\ref{sec:operations} defines a rounding procedure for the TR-format and redefines some common linear algebra operations in order to make them amenable to storage-cost savings using the rounding procedure. Sec.~\ref{sec:graph} extends the procedures to general graph-based formats and contains algorithms for converting representations of tensors between different graph-based formats. Sec.~\ref{sec:examples} concludes with a number of performance comparisons of the presented algorithms with algorithms previously described in the literature. Implementations of all algorithms in this paper are publicly available online\footnote{\href{https://github.com/oscarmickelin/tensor-ring-decomposition}{\texttt{https://github.com/oscarmickelin/tensor-ring-decomposition}}}.

\subsection{Notation}\label{sec:notation}
Algorithms will be presented in pseudocode using MATLAB commands and notation. We will mostly follow the notation of Hackbush\cite{hackbusch2012tensor}. The $i$:th standard basis vector will be denoted by $e_i$. We will use the Frobenius tensor norm $\| \cdot \|_F$ given by $\| T \|_F^2 := \sum_{i_1, \ldots , i_d} \abs{T(i_1, \ldots , i_d)}^2.$ The $k$:th unfolding matrix of a tensor $T\in \mathbb{R}^{n_1\times  \cdots  \times n_d}$ will be denoted by $T_{\langle k\rangle} \in \mathbb{R}^{ \left( \prod_{j=1}^kn_j \right) \times \left( \prod_{j=k+1}^d n_j \right)}$, and can be computed by $\text{reshape}\bigl(T, \bigl[\prod_{j=1}^kn_j, \prod_{j=k+1}^d n_j\bigr]\bigr)$ in MATLAB. The $\delta$-rank of a matrix $A$ is $\text{rank}_{\delta}(A) := \min_{B: \|A-B\|_F \leq \delta} \text{rank}(B)$ and can be computed by calculating the rank of the result of a $\delta$-truncated SVD on $A$, $\text{SVD}_{\delta}(A)$. The $k$-mode product of a tensor $T$ with a matrix $A\in \mathbb{R}^{\ell \times n_k}$ is a tensor $T\times_k A$ in $\mathbb{R}^{n_1 \times  \cdots  \times \ell \times \cdots  \times n_d}$ defined by $\left(T\times_k A \right)(i_1, \ldots , i_{k-1} , j, i_{k+1} , \ldots , i_d) := \sum_{i_k=1}^{n_k} T(i_1, \ldots , i_d) A(j, i_k).$ The Hadamard (or elementwise) product of two tensors $T_1, T_2$ in $\mathbb{R}^{n_1 \times  \cdots  \times n_d}$ is defined by $\left(T_1 \circ T_2\right) (i_1, \ldots , i_d) := T_1(i_1, \ldots , i_d)T_2(i_1, \ldots , i_d).$

We will denote the group of circular shifts on $d$ variables by $\text{S}_d^c$, which has generator $\gamma = (1, d, d-1, \ldots, 2)$ in cycle notation \cite[p.~30]{lang2002algebra}. This notation means that $\gamma(1) = d, \gamma(d) = d-1 , \ldots , \gamma(2) = 1$, i.e., $\gamma$ corresponds to a circular shift to the left by one step. The inverse shift $\gamma^{-1}$ therefore corresponds to a circular shift to the right by one step, and can be written in cycle notation as $\gamma^{-1} = (2,3, \ldots , d, 1)$.

For any $\tau \in \text{S}_d^c$ and tensor $T \in \mathbb{R}^{n_1 \times  \cdots \times n_d}$, we define the \emph{$ \tau$-permuted tensor} $T^{\tau} \in   \mathbb{R}^{n_{\tau^{-1}(1)} \times \cdots \times n_{ \tau^{-1}(d)}}$ by $T^{ \tau}(i_1, \ldots , i_d) = T(i_{\tau(1)}, \ldots , i_{\tau(d)})$. Tensor products will be denoted by $\cdot \otimes \cdot$, and since we will exclusively deal with finite-dimensional vector spaces, we will explicitly work with coordinate representations and we identify $\left(v_1 \otimes v_2 \otimes \cdots \otimes v_d\right)(i_1, i_2 , \ldots , i_d) := v_1(i_1)\cdot v_2(i_2)\cdot \ldots \cdot v_d(i_d),$ for $v_k \in \mathbb{R}^{n_k}$. The Kronecker product of two matrices $A, B$ will be denoted by $A\otimes_{\text{K}} B$, to avoid confusion.

For graph-based tensor formats, there is an elegant and coordinate-free notation \cite{landsberg2011geometry,ye2018tensor} which avoids explicit reference to cores and indices. However, we will be concerned with practical algorithms, which therefore are required to work explicitly with the cores and indices which arise in the implementation data structures. Our notation is therefore chosen to reflect this. Let $\mathcal{G} = (V,E)$ be an undirected graph with the set of vertices $V = \{1, \ldots , d\}$ and with the set of edges $E$. We will denote an edge between vertices labeled by $k_1$ and $k_2$ by $(k_1,k_2)$. The set of edges emanating from the vertex $k$ will be denoted by $e(k)$. For each $i\in e(k)$, we will specify a maximal edge rank $r_{ik}\in \mathbb{N}$. A tensor $T \in \mathbb{R}^{n_1\times  \cdots  \times n_d}$ will then be said to be in $\mathcal{G}$-format if there exist core tensors for each vertex, denoted by $G_k(i_k, \bigtimes_{i \in e(k)} \alpha_{ik})$ where $1 \leq \alpha_{ik} \leq r_{ik}$, such that $T(i_1, \ldots , i_d)$ equals the contraction over all $\alpha_{ik}$ of $G_k(i_k, \bigtimes_{i \in e(k)} \alpha_{ik})$.  We will then refer to the collection of $r_{ik}$ for $1\leq k \leq d$ and $i\in e(k)$ as the $\mathcal{G}$-ranks of $T$. See e.g., the recent review article by Or{\'u}s \cite{orus2014practical} for more information. In the case when $\mathcal{G}$ is a cycle of $d$ vertices, the $\mathcal{G}$-format is precisely the TR-format.

\section{Conversion from full format to TR-format}\label{sec:heur}
This section describes the conversion of a tensor in full format into TR-format. A negative result in Sec.~\ref{sec:neg_result} (Prop.~\ref{prop:incomparable}) first describes the importance of the choice of $r_0$ and how the situation for the TR-format differs from the TT-format. Afterwards in Sec.~\ref{sec:cyclic_shifts}, we describe how considering cyclic shifts $\tau \in \text{S}_d^c$ gives an additional degree of freedom to be used in the conversion into TR-format. Both $r_0$ and the cyclic shift $\tau$ can be chosen via exhaustive search which gives a reduced storage-cost algorithm (Alg.~\ref{alg:reduced_storage_TR-SVD}) in Sec.~\ref{sec:check_all}. Finally, in Sec.~\ref{sec:choicedivperm} we present an algorithm to produce a low-cost choice of $r_0$ and $\tau$ (Alg.~\ref{alg:reducedTR-SVD}).

\subsection{Negative result}\label{sec:neg_result}
As presented by Zhao et al. \cite{zhao2017learning,zhao2016tensor}, Alg.~\ref{alg:TR-SVD} below can be used to compute an approximate TR-representation $\widetilde{T}$ of a tensor $T$ given in full format, given a desired accuracy $\varepsilon$. $\widetilde{T}$ then satisfies $\| T - \widetilde{T} \|_F \leq \varepsilon\|T\|_F$. Alg.~\ref{alg:TR-SVD} computes the quantity $\delta := \frac{\varepsilon \|T\|_F}{\sqrt{d}}$ and requires a manual input of a divisor $r_0$ of $\text{rank}_{\delta}\left( T_{\langle 1 \rangle}\right)$. The choice of $r_0$ by Zhao et al. \cite{zhao2017learning,zhao2016tensor} (and also for a related algorithm based on the skeleton/cross approximation\cite{espig2012note}), is to minimize $ \abs{r_0 - \frac{\text{rank}_{\delta}\left( T_{\langle 1 \rangle}\right)}{r_0}}$, but examples (see Sec.~\ref{sec:ex_full2TR}) show that this can lead to suboptimal compression ratios.
\begin{algorithm}
\caption{TR-SVD \cite{zhao2017learning,zhao2016tensor}}\label{alg:TR-SVD}
\begin{algorithmic}[1]
 \Require{$d$-tensor $T$ in full format, accuracy $\varepsilon$, divisor $r_0$ of $\text{rank}_{\delta}\left( T_{\langle 1 \rangle}\right)$.}{}
 \Ensure{Core tensors $G_1$, \ldots , $G_d$ s.t. $\widetilde{T}$ of form in Eq.~\eqref{eq:TRdecomp} has $\| T - \widetilde{T} \|_F \leq \varepsilon\|T\|_F$.}{} \\
 Compute $\delta := \frac{\varepsilon \|T\|_F}{\sqrt{d}}$.\\
 $C = \text{reshape}(T, [n_1, \frac{\text{numel}(T)}{n_1}]) $ \Comment{Initial step}\\
 $[U, \Sigma, V] = \text{SVD}_{\delta}(C)$ \\
 Put $r_1 := \text{rank}\left(\Sigma \right)$.\\
$G_1 = \text{permute}(\text{reshape}(U, [n_1, r_0, r_1]), [2, 1, 3])$ \\
 $C :=  \text{permute}(\text{reshape}(\Sigma V^T, [r_0, r_1, \prod_{j=2}^dn_j]), [2, 3, 1])$ \\ 
Merge the last two indices by $C = \text{reshape}(C, [r_1, \prod_{j=2}^{d-1}n_j, n_dr_0])$.
\For{$k =2:d-1$} \Comment{Main loop}
\State         $C = \text{reshape}(C, [r_{k-1} n_k, \frac{\text{numel}(C)}{(r_{k-1} n_k)}])$     
\State         $[U,S,V] = \text{SVD}_{\delta}(C)$
\State         $r_k = \text{rank}(\Sigma)$
\State         $G_k= \text{reshape}(U, [r_{k-1}, n_k, r_k])$
\State         $C = \Sigma V^T$
\EndFor \\
$G_d = \text{reshape}(C, [r_{d-1}, n_d, r_0])$ \Comment{Final step}
 \end{algorithmic}
\end{algorithm}

In Sec.~\ref{sec:check_all}, we will minimize storage costs of a TR-representation of a tensor by choosing $r_0$ appropriately. It is therefore of interest to compare the rank vectors $r$ and $r'$ resulting from the different choices $r_0$ and $r_0'$ in Alg.~\ref{alg:TR-SVD}. To this end, we will make use of the following concept. We will say that a vector $r := (r_0, \ldots , r_{d-1}, r_0)$ is a minimal rank of $T$ if (i) there exists a TR-representation of $T$ with TR-ranks $r$, and (ii) no other TR-rank $r'$ of $T$ satisfies $r' \leq r$ under the elementwise inequality in $\mathbb{R}^{d+1}$. The elementwise inequality is only a partial order on $\mathbb{R}^{d+1}$, so there can be multiple minimal ranks for a given tensor \cite{ye2018tensor}; Prop.~\ref{prop:incomparable} below will show that this is in fact a common occurrence. For the TT-representation \cite{oseledets2011tensor}, the ranks satisfy $r_k \geq \text{rank}\left(T_{\langle k \rangle}\right)$ and an exact TT-decomposition using the TT-SVD algorithm results in ranks satisfying $r_k = \text{rank}\left(T_{\langle k \rangle}\right)$. This is therefore the unique minimal rank with $r_0 = 1 = r_d$. For the TR-format, an analogous argument \cite{zhao2016tensor} shows that $r_0r_k \geq \text{rank}\left( T_{\langle k \rangle}\right)$, which will be used below.

Clearly, smaller rank vectors under the elementwise ordering result in lower storage cost. Rank vectors that are a priori known to be greater than others under the elementwise ordering can therefore be disregarded. Unfortunately, this situation does not occur when using Alg.~\ref{alg:TR-SVD}, as we show in Prop.~\ref{prop:incomparable}.

\begin{lemma}\label{lemma:incomparable}
Let $r_0$ divide $\text{rank}\left( T_{\langle 1 \rangle}\right)$. If $r$ denotes the rank obtained when running Alg.~\ref{alg:TR-SVD} with precision $\varepsilon = 0$ and first rank $r_0$, then any minimal rank $r' \leq r$ has $r_0' = r_0$ and $r_1' = r_1$.
\end{lemma}
\begin{proof}
The minimal rank obeys $r_0' \leq r_0$ and $r_1' \leq r_1$, by definition and $r_0'r_1' \geq \text{rank}\left( T_{\langle 1 \rangle}\right) = r_0r_1$, where equality holds by construction of the TR-SVD algorithm. The conclusion follows.
\end{proof}

\begin{proposition}\label{prop:incomparable}
Let $r_0, r_0'$ be two distinct divisors of $\text{rank}\left(T_{\langle 1 \rangle}\right)$. If $r$ and $r'$ denote the ranks obtained when running Alg.~\ref{alg:TR-SVD} with precision $\varepsilon = 0$ and first rank $r_0$ and $r_0'$, respectively, then there is no common minimal rank $r_m$ satisfying $r_m \leq r$, and $r_m \leq r'$.
\end{proposition}
\begin{proof}
This follows from Lemma~\ref{lemma:incomparable} and the fact that $r_0 \neq r_0'$.
\end{proof}
Different choices of divisors $r_0$ and $r_0'$ are therefore incomparable a priori, so to find the rank which results in the lowest storage cost for a given accuracy, we need to compare the results of using all choices of divisor $r_0$ in Alg.~\ref{alg:TR-SVD}. This also implies that it is of interest to be able to convert a TR-representation with initial rank $r_0$ to one with initial rank $r_0' \neq r_0$, and this is considered in Sec.~\ref{sec:transforming_graph} below.

\subsection{Cyclic shifts}\label{sec:cyclic_shifts}
We next consider cyclic shifts to reduce storage costs further. Let $\tau\in \text{S}_d^c $ be a cyclic shift. If $T$ is given in the form of Eq.~\eqref{eq:TRdecomp}, then $T^{\tau}$ can be represented by
\begin{align}
\begin{split}T^{\tau}(i_1, \ldots , i_d) &= \sum_{\alpha_0 = 1 }^{r_{0}}\cdots \sum_{\alpha_{d-1} = 1 }^{r_{d-1}}  G_1(\alpha_0, i_{\tau(1)}, \alpha_1) \cdot \ldots \cdot G_d(\alpha_{d-1}, i_{\tau(d)}, \alpha_0) \\
&= \sum_{\alpha_0 = 1 }^{r_{0}}\cdots \sum_{\alpha_{d-1} = 1 }^{r_{d-1}}   G_{\tau^{-1}(1)}(\alpha_0, i_1, \alpha_1) \cdot \ldots \cdot G_{\tau^{-1}(d)}(\alpha_{d-1}, i_d, \alpha_0),
\end{split}\end{align}
using the fact that the trace of a product of matrices is unchanged under cyclic shifts of the matrices. $T^{\tau}$ is then of the format in Eq.~\eqref{eq:TRdecomp} with core tensors $G_{\tau^{-1}(1)}, \ldots , G_{\tau^{-1}(d)}$. Conversely, we can fix a cyclic shift $\tau$ and compute a TR-representation of $T^{\tau}$ with cores $G_k(i_{\tau^{-1}(k)})$. A TR-representation of $T$ is then given by the cores $G_{\tau(1)}(i_1), \ldots ,G_{\tau(d)}(i_d)$. We will see in the following that choosing the cyclic shift $\tau$ appropriately and computing a TR-representation in this way can result in far lower storage costs than without the use of the cyclic shift. The following is an illustrative example.

\begin{example}\label{ex:why_shift}
Let $T \in \mathbb{R}^{n_1\times \cdots  \times n_d}$ be a tensor given as a discretization of a function $f: \mathbb{R}^d \rightarrow \mathbb{R}$ that, in fact, does not depend on $x_2, \ldots , x_{d-1}$. Consider a cyclic shift of the indices of $T$. Specifically, study $T^\tau \in \mathbb{R}^{n_d\times n_1 \times  \cdots \times n_{d-1}}$, for $\tau = (1, d, d-1, \ldots, 2)^{-1}$. For $k\geq 2$, any unfolding matrix $T^{\tau}_{\langle k \rangle} \in \mathbb{R}^{ (n_d n_1  \ldots n_{k-1} ) \times (n_{k} \ldots  n_{d-1})}$ has constant columns, so matrix rank $1$. With $r_0 = 1$, Alg.~\ref{alg:TR-SVD} results in $r_k = 1$ for $k \geq 2$ and consequently a low storage cost. Directly computing a TR-representation of $T$ without making use of a cyclic shift will in general give ranks $r_k > 1$, incurring far higher storage cost. In other words, an appropriate choice of cyclic shift can result in great storage savings.
\end{example}

Note that cyclic shifts cannot be used in the same way when considering the TT-format. Storage costs of individual tensors can be reduced by fixing one given shift of the indices. However, operations in the TT-format such as addition can then only be performed on the two tensors if they use the same shifts. This makes this approach impractical for the TT-format. It should also be noted that this approach is not valid for general permutations, since for permutations $\tau$ that are not cyclic shifts, there is in general a $k$ such that $\alpha_{\tau^{-1}(k)} \neq \alpha_{\tau^{-1}(k+1)}$.


\subsection{Algorithm with lower storage cost}\label{sec:check_all}
The previous two sections provide an algorithm lowering the total storage cost, written out in Alg.~\ref{alg:reduced_storage_TR-SVD}. It consists of an outer loop over each cyclic shift $\tau \in \text{S}_d^c$ and an inner loop over each divisor $r_0$ of rank$_{\delta}\left( T^{\tau}_{\langle 1 \rangle} \right)$. For each choice of $(\tau, r_0)$, Alg.~\ref{alg:TR-SVD} can be used to compute a tensor representation, and the representation with the lowest storage cost can be retained.

\begin{algorithm}
\caption{Reduced storage TR-SVD}\label{alg:reduced_storage_TR-SVD}
\begin{algorithmic}[1]
 \Require{Full $d$-tensor $T$, accuracy $\varepsilon$.}{}
 \Ensure{Core tensors $G_1$, \ldots , $G_d$ s.t. $\widetilde{T}$ of form in Eq.~\eqref{eq:TRdecomp} has $\| T - \widetilde{T} \|_F \leq \varepsilon\|T\|_F$.}{} 
 \For{$k = 1:d$} 
 \State Put $\tau_k = ( 1, d, d-1, \ldots, 2  )^{k-1}$
\For{each $r_0$ divisor of $\text{rank}_{\delta}\left( T^{\tau_k}_{\langle 1 \rangle}\right)$}
 \State Call Alg.~\ref{alg:TR-SVD} on $T^{\tau_k}$ with first rank $r_0$
 \State Store $T^{\tau_k}$ if its TR-representation has the lowest storage cost so far
 \EndFor 
 \EndFor 
 \end{algorithmic}
\end{algorithm}

This leads to a total complexity which is $C(d)$ times the cost of a single call to Alg.~\ref{alg:TR-SVD}, where $C(d) = \sum_{ \tau_k \in \text{S}_d^c } \left(\text{number of divisors of } \text{rank}_{\delta} \left( T^{\tau_k}_{\langle 1 \rangle} \right)\right)  \thicksim \mathcal{O}\left( d\cdot \text{div}(r)\right),$ if $\text{rank}_{\delta}\left(T^{\tau_k}_{\langle 1 \rangle} \right) \thicksim r$ for all $k$. Here, $\text{div}(r)$ denotes the number of divisors of a positive integer $r$. The asymptotic complexity is then $\mathcal{O}(dr^{1/\text{log} \text{log}(r)})$ times that of a call to Alg.~\ref{alg:TR-SVD}. This procedure requires keeping the currently smallest (which can be significantly larger than the actually smallest) representation stored during the entire duration of the algorithm. It can, however, be run in parallel on up to $d$ processors simultaneously.

\subsection{Heuristic reduced-cost algorithm} \label{sec:choicedivperm}
It would be desirable to determine a choice $(\tau, r_0)$ leading to low storage cost without exhausting all possibilities. In this section, we describe a heuristic approach to this choice with small extra computational cost compared to Alg.~\ref{alg:TR-SVD}.
\subsubsection{Heuristic choice of cyclic shift}
Let $r_k\left(T\right)$ denote the TR-ranks of the tensor $T$ as computed by Alg.~\ref{alg:TR-SVD} for some choice of $r_0$ and $\varepsilon$ that will be clear from the context. Note that the storage cost of a tensor in TR-format is given by the sum $\sum_k r_{k-1}r_kn_k$, so a low storage cost can be achieved by minimizing $\max_k r_k\left( T^{\tau}\right)$ over $\tau \in \text{S}_d^c$. This would however require a sweep over all indices $k$ for each choice of $\tau$ and therefore incur the same cost as computing the TR-SVD decompositions of all $T^{\tau}$. However, the TT-ranks $r_k$ as computed by the TT-SVD algorithm are typically non-decreasing up to an index $k_*$, after which they are non-increasing. As a less costly alternative to minimizing $\max_k r_k\left( T^{\tau}\right)$ over $\tau \in \text{S}_d^c$, one can then instead minimize $r_2\left( T^{\tau}\right)$, in an attempt to minimize the slope of $r_k\left( T^{\tau}\right)$ as a function of $k$, and therefore minimize its maximum value. We therefore make the following definition.

\begin{definition} The \emph{$k$:th interaction matrix} of a tensor $T \in \mathbb{R}^{n_1\times  \cdots  \times n_d}$ is denoted by
\begin{equation}
\mathcal{IM}_k(T) \in \mathbb{R}^{ (n_kn_{k+1})\times (n_{k+2}\ldots n_d n_1 \ldots n_{k-1})}.
\end{equation}
It is given by $T^{\tau_k}_{\langle 2 \rangle}$, where $\tau_k = ( 1, d, d-1, \ldots, 2  )^{k-1}$. The \emph{$k$:th interaction rank} of $T$ is defined to be $ir_k (T) := \text{rank}\left(\mathcal{IM}_k(T)\right).$
\end{definition}

We will see in the examples below that choosing
\begin{equation}\label{eq:kstar}
k_* = \underset{k = 1, \ldots , d}{\argmin} ir_k(T).
\end{equation}
and calling TR-SVD on $T^{\tau_{k_*}}$ results in low storage cost.

\subsubsection{Heuristic choice of divisor}
After choosing $k_*$, we have the choice of choosing a divisor $r_0$ of $ \text{rank}_{\delta}\left( T^{\tau_{k_*}}_{\langle 1 \rangle}\right)$ when running TR-SVD above. As an alternative to looping over all possible divisors, we can choose the divisor to be close to the interaction ranks. In detail, we choose $r_{0}^{*} \mid \text{rank}_{\delta}\left( T^{\tau_{k_*}}_{\langle 1 \rangle} \right)$ as 

\begin{equation}\label{eq:pstar}
r_0^* = \underset{r_0 \mid \text{rank}_{\delta}\left( T^{\tau_{k_*}}_{\langle 1 \rangle}\right)}{\argmin} \Big| ir_{k_*-1}(T) -  \frac{\text{rank}_{\delta} T^{\tau_{k_*}}_{\langle 1 \rangle} }{r_0} \Big|+ \Big|ir_{k_*}(T)-r_0 \Big|.
\end{equation}

The intuition behind this is that $r_0$ controls the strength of interaction between $i_{k_*-1}$ and $i_{k_*}$ in the TR-format, since larger values of $r_0$ accommodate a higher number of components of the corresponding core tensors $G_{k_*-1}$ and $G_{k_*}$. Likewise, $\frac{\text{rank}_{\delta}\left(T^{\tau_{k_*}}_{\langle 1 \rangle} \right)}{r_0}$ equals $r_1$ for the TR-representation of of $T^{\tau_{k_*}}$. It therefore measures the strength of interaction between $i_{k_*}$ and $i_{k_*+1}$.

Another measure of the strength of interaction between $i_{k_*-1}$ and $i_{k_*}$ in the full tensor $T$ is $ir_{k_*} (T)$, since incoherence in these dimensions implies a large rank for the matrix $\mathcal{IM}_{k_*}(T)$, as in Example~1. Likewise, a strong interaction between $i_{k_*}$ and $i_{k_*+1}$ would result in a large rank of the matrix $\mathcal{IM}_{k_*-1}(T)$. We therefore choose $r_0$ to match these two measures of interaction as well as possible.

\subsubsection{Heuristic algorithm}
We combine the steps in the previous two sections into an algorithm, which then consists of
\begin{itemize}
\item a preprocessing step wherein the interaction ranks $ir_k(T)$ are computed for each $k=1,2,\ldots , d$.  We choose $k_*$ from Eq.~\eqref{eq:kstar} and subsequently $r_0^*$ from Eq.~\eqref{eq:pstar}. This results in a choice $(\tau_{k_*}, r_0^*)$.
\item one call of Alg.~\ref{alg:TR-SVD} on $T^{\tau_{k_*}}$, using the divisor $ r_0^*$ of $ \text{rank}_{\delta}\left(T^{\tau_{k_*}}_{\langle 1 \rangle} \right)$ producing cores $\widetilde{G}_k(i_{\tau^{-1}_{k_*}(k)})$. A TR-representation of $T$ is then given by the cores $\widetilde{G}_{\tau_{k_*}(k)}(i_k)$.
\end{itemize}
We summarize this in Alg.~\ref{alg:reducedTR-SVD}.

\begin{algorithm}
\caption{Heuristic TR-SVD}\label{alg:reducedTR-SVD}
\begin{algorithmic}[1]
 \Require{Full $d$-tensor $T$, accuracy $\varepsilon$.}{}
 \Ensure{Core tensors $G_1$, \ldots , $G_d$ s.t. $\widetilde{T}$ of form in Eq.~\eqref{eq:TRdecomp} has $\| T - \widetilde{T} \|_F \leq \varepsilon\|T\|_F$.}{} \\
Compute $\mathcal{IM}_k(T)$, $ir_k(T)$ for $k = 1, \ldots , d$. \Comment{Preprocessing step}\\
Choose $k_*$ from Eq.~\eqref{eq:kstar} and subsequently $r_0^*$ from Eq.~\eqref{eq:pstar}. \\
Run Alg.~\ref{alg:TR-SVD} on $T^{\tau_{k_*}}$, using the divisor $ r_0^*$ of $ \text{rank}_{\delta}\left( T^{\tau_{k_*}}_{\langle 1 \rangle} \right)$ to obtain cores $\widetilde{G}_k(i_{\tau_{k_*}^{-1}(k)})$.\\
Set $G_k(i_k) = \widetilde{G}_{\tau_{k_*}(k)}(i_k)$.
 \end{algorithmic}
\end{algorithm}

Since $\text{rank}\left(T^{\tau_{k_*}}_{\langle 1 \rangle} \right) \leq n_{k_*}$, the cost of the preprocessing step of Alg.~\ref{alg:reducedTR-SVD} is 
\begin{equation}
\mathcal{O}\left( (n_1n_2 + n_2n_3 + \ldots + n_dn_1) \prod_{i=1}^dn_i + \text{div}(n_{k_*})\right) \thicksim \mathcal{O} \left( dn^{d+2} + n^{\frac{1}{\text{log log}(n)}}\right) ,
\end{equation}
if $n_i \thicksim n$ for all $i$. Comparisons of performance in terms of actual storage cost and total computation times are performed in Sec.~\ref{sec:ex_full2TR}.

\section{Operations}\label{sec:operations}
This section details some common operations performed on the TR-format. We firstly define an SVD-based rounding procedure for the TR-format (Alg.~\ref{alg:TR-truncation} in Sec.~\ref{sec:truncation}), and then find it necessary to redefine the addition of two TR-representations in order to make full use of the rounding procedure (Thm.~\ref{thm:add} and Prop.~\ref{prop:negrounding} in Sec.~\ref{sec:addition}). We finally comment on the computation of the Frobenius norm of a tensor in Sec.~\ref{sec:frob}. For the remainder of this section, we will let $T$ be a tensor for which a TR-decomposition of the form in Eq.~\eqref{eq:TRdecomp} has been computed.
\subsection{Rounding}\label{sec:truncation}
Analogously to the case of the TT-format \cite{oseledets2011tensor}, we can perform a rounding of a TR-representation with suboptimal ranks $r_0, r_1, \ldots , r_d$. The proposed procedure is, just like the TT-rounding procedure \cite[Alg.~2]{oseledets2011tensor}, based on a structured QR-decomposition which enables lower costs of the SVD-computations in Alg.~\ref{alg:TR-SVD}, and results in a TR-representation with ranks $\widetilde{r}_0 , \widetilde{r}_1, \widetilde{r}_2, \ldots , \widetilde{r}_{d-1}, \widetilde{r}_d =  \widetilde{r}_0$. Unlike for the TT-rounding procedure, invariance under cyclic shifts of the successive matrix factors in the QR-decompositions of the reshaped cores contribute to an even lower storage cost. Also, the motivation for the TT-rounding procedure is that it carries out the same steps as the TT-SVD algorithm. This is not easily generalized to the TR-format - the reshaping of the low-rank decomposition of the first unfolding matrix $T_{\langle 1 \rangle}$ in Alg.~\ref{alg:TR-SVD} cannot be captured on the level of cores in a straightforward way. This necessitates a different approach, which is detailed in Alg.~\ref{alg:TR-truncation}. Unlike for TT-SVD, the TR-SVD algorithm has no guarantee of quasi-optimality since the TR-format is not closed. Instead, only a bound on the approximation error is guaranteed. It is therefore enough to produce a rounding procedure with just a bound on the approximation error, which then does not necessitate the procedure to perform the same steps as the TR-SVD algorithm.

\begin{algorithm}
\caption{TR-rounding}\label{alg:TR-truncation}
\begin{algorithmic}[1]

 \Require{$d$-tensor $T$ with cores $G_1, \ldots , G_d$, TR-ranks $r_0, \ldots , r_d$, accuracy $\varepsilon$.}{}
 \Ensure{$d$-tensor $\widetilde{T}$ with cores $\widetilde{G}_k$, TR-ranks $\widetilde{r}_k \leq r_k$ and $\| T - \widetilde{T} \|_F \leq \varepsilon\|T\|_F$.}{} \\
 Compute the core-level accuracy $\delta = \frac{\varepsilon\|T\|_F}{\sqrt{dr_0}} $
 \For{$k = 1:d-1$} \Comment{Structured QR-sweep}
 \State $[Q_k, R_k] = \text{QR}( \text{reshape}(G_k, [r_{k-1}n_k, r_k]))$
 \State $G_k =  \text{reshape}(Q_k, [r_{k-1}, n_k, \text{numel}(Q_k)/(r_kn_k)]))$
 \State $G_{k+1} = G_{k+1} \times_1 R_k$
 \EndFor \\
$[Q_d, R_d] = \text{QR}( \text{reshape}(G_d, [r_{d-1}n_d, r_d]))$\Comment{Cyclic step for reducing end-ranks}\\
$[U_d, \Sigma_d, V_d] = \text{SVD}_{\delta}(R_d)$\\
$\widetilde{r}_d = \text{rank}(\Sigma_d)$
\If{$\widetilde{r}_d < r_d$}
\State $G_d = \text{reshape}(Q_dU_d\Sigma_d, [r_{d-1}, n_d,\widetilde{r}_d] )$
\State $G_1 = G_1 \times_1 V_d^T$
\EndIf 
\For{$k =d:-1:2$} \Comment{SVD-sweep for reducing remaining ranks}
\State         $[U_k,\Sigma_k,V_k] = \text{SVD}_{\delta}(\text{reshape}(G_k, [r_{k-1}, n_kr_k]))$
\State         $\widetilde{r}_{k-1} = \text{rank}(\Sigma_k)$
\State	 $\widetilde{G}_k = \text{reshape}(V_k^T, [\widetilde{r}_{k-1}, n_k, \widetilde{r}_{k}])$
\State         $G_{k-1}= G_{k-1}\times_3\text{reshape}(U_k\Sigma_k, [r_{k-1}, n_k\widetilde{r}_{k-1}])^T$
\EndFor \\
$\widetilde{G}_1 = G_1$
 \end{algorithmic}
\end{algorithm}

We now prove the correctness of Alg.~\ref{alg:TR-truncation}.
\begin{theorem}
Given a tensor $T$ with cores $G_k$, Alg.~\ref{alg:TR-truncation} returns cores $\widetilde{G}_k$ with corresponding tensor $\widetilde{T}$ satisfying $\| T- \widetilde{T}\|_F \leq \varepsilon\|T\|_F$.
\end{theorem}

\begin{proof}
The main insights in the proof are that (i) rounding of a TR-representation of $T$ can be related to the rounding of a related TT-representation by reshaping $T$ to have first component of size $r_0n_1$;  (ii) introducing a cyclic step in Alg.~\ref{alg:TR-truncation} enables a reduction of the end rank $r_0$ and this can be encoded by introducing a helper index $i_{d+1}$ with special structure.

We will relate Alg.~\ref{alg:TR-truncation} to an application of the TT-rounding procedure on an extended tensor $T_e$ in $\mathbb{R}^{r_0n_1\times n_2 \times  \cdots  \times n_{d-1} \times n_d \times n_{d+1}}, $with $n_{d+1} = r_d$ and $r_{d+1}=1$. The entries of $T_e$ are defined by its TT-representation $T_e(\alpha_0i_1, i_2 , \cdots ,i_d, i_{d+1}) = \sum_{\alpha_1, \ldots , \alpha_{d}}\prod_{k=1}^{d+1} G_k(\alpha_{k-1}, i_k, \alpha_k),$ where the final core is $G_{d+1}(i_{d+1}) = e_{i_{d+1}} \in \mathbb{R}^{r_d\times 1}$. Note here that the index $\alpha_0$ is not summed over. Using the Kronecker delta defined as $\delta_{i,j} = 1$ if $i=j$ and $0$ otherwise, we have
\begin{equation}\label{eq:firstSum}
\begin{split}
\sum_{\alpha_0, i_{d+1}=1}^{r_d}   T_e(\alpha_0i_1, i_2 , \cdots ,i_d, i_{d+1})\delta_{\alpha_0, i_{d+1}}
&=   \sum_{\alpha_0, \alpha_1, \ldots , \alpha_{d}}\left( \prod_{k=1}^{d} G_k(\alpha_{k-1}, i_k, \alpha_k)\right)\cdot e_{\alpha_0}\\
&= \text{Trace}\left(\sum_{\alpha_1, \ldots , \alpha_{d-1}}\prod_{k=1}^{d} G_k(\alpha_{k-1}, i_k, \alpha_k)\right)= T(i_1, \ldots , i_d),
\end{split}
\end{equation}
which will be used below. Next, a TT-rounding \cite[Alg.~2]{oseledets2011tensor} of $T_e$ with precision $\frac{\varepsilon\|T\|_F}{\sqrt{dr_0}}$ (instead of with $\sqrt{d-1}$, since $T_e$ has $d+1$ dimensions) results in a tensor $\widehat{T}_e$ with cores $\widehat{G}_k$ satisfying
\begin{equation}\label{eq:exterror}
\| T_e - \widehat{T}_e\|_F \leq \frac{\varepsilon\|T\|_F}{\sqrt{r_0}}.
\end{equation}
By the choice of the final core $G_{d+1}$, Alg.~\ref{alg:TR-truncation} now returns precisely the cores
\begin{equation}
\widetilde{G}_k = \begin{cases}
\widehat{G}_k, &k \neq 1, \\
\text{reshape}(\widehat{G}_1, [r_0, n_1, \widetilde{r}_1]) \times_1 \widehat{G}^{ T}_{d+1}, & k = 1,
\end{cases}
\end{equation}
where the last equation ignores the singleton dimension $r_{d+1} = 1$ in $\widehat{G}_{d+1}$. This can be verified by writing out the steps in \cite[Alg.~2]{oseledets2011tensor} and using the fact that $\text{reshape}\left(G_{d+1}, [r_d, n_{d+1}]\right) = I_{r_d}$, by construction. This means that
\begin{equation}\label{eq:secondSum}
\begin{split}
\sum_{\alpha_0, i_{d+1}=1}^{r_d}    \widehat{T}_e(\alpha_0i_1, i_2 , \cdots ,i_d, i_{d+1})\delta_{\alpha_0, i_{d+1}} &=   \sum_{\alpha_0, \alpha_1, \ldots , \alpha_{d}} \left( \prod_{k=1}^{d} \widehat{G}_k(\alpha_{k-1}, i_k, \alpha_k)\right)  \widehat{G}_{d+1}(\alpha_d, \alpha_0, 1) \\
&= \text{Trace}\left(\sum_{\alpha_1, \ldots , \alpha_{d-1}}\prod_{k=1}^{d} \widetilde{G}_k(\alpha_{k-1}, i_k, \alpha_k)\right)= \widetilde{T}(i_1, \ldots , i_d),
\end{split}
\end{equation}
which we now use. Write $\Delta T := T - \widetilde{T}$, and $\Delta T_e := T_e - \widehat{T}_e$. Eqs.~\eqref{eq:firstSum} and \eqref{eq:secondSum} now give
\begin{align}
\begin{split}
\| \Delta T \|_F^2 &=  \sum_{i_1, \ldots , i_d}     \abs{\Delta T(i_1, \ldots, i_d)}^2 =   \sum_{i_1, \ldots , i_d}   \Big|  \sum_{\alpha_0, i_{d+1} = 1}^{r_d}    \Delta T_e(\alpha_0 i_1, \ldots, i_{d+1})\delta_{\alpha_0, i_{d+1}}\Big|^2 \\
&\leq r_0   \sum_{\alpha_0, i_1, \ldots , i_{d}}     \abs{\Delta T_e(\alpha_0 i_1, \ldots, \alpha_0)}^2 \leq r_0     \sum_{\alpha_0, i_1, \ldots , i_{d+1}}    \abs{\Delta T_e(\alpha_0 i_1, \ldots, i_{d+1})}^2 \label{eq:notTight}  \leq  \varepsilon^2 \|T\|_F^2,
\end{split}
\end{align}
where we used Eq.~\eqref{eq:exterror} in the last step and Cauchy-Schwarz in the third step.
\end{proof}

\begin{remark}
The precision $\frac{\varepsilon}{\sqrt{dr_0}}$ is usually overly conservative in practice, stemming from the fact that the second inequality in Eq.~\eqref{eq:notTight} is generically far from being tight, since the left hand side sums over a factor of $r_0$ fewer elements than the right hand side.
\end{remark}

The complexity of Alg.~\ref{alg:TR-truncation} is $\mathcal{O}\left( dnr^3\right)$ plus the complexity of calculating the norm $\|T\|_F$. We will see below that it is possible to compute $\|T\|_F$ with complexity $\mathcal{O}\left( \min_k r_k dnr^3\right)$. This reduces to the cost of computing the Frobenius norm in the TT-format\cite{oseledets2011tensor} when $\min_k r_k = 1$, i.e., in particular when $r_0 = 1 = r_d$.

\subsection{Addition}\label{sec:addition}
One criterion for the rounding procedure in Alg.~\ref{alg:TR-truncation} to be of practical use, is for it to reduce the ranks of the added tensor $T+T$ to those of $T$. For this to hold, we need to redefine the addition operation described in earlier literature, as in the following.
\begin{theorem}\label{thm:add}
Let $T'$ and $T''$ be tensors having TR-representations with cores $G_k', G_k''$ for $k=1, \ldots , d$. The tensor $T = T' + T''$ can be written in TR-format with cores $G_k$ where
\begin{equation}\label{eq:addition}
G_k(i_k) = \begin{bmatrix}
G_k' (i_k)& 0 \\ 0 &  G_{k}''(i_k)
\end{bmatrix},
\end{equation}
for $k = 2, \ldots , d-1$. If $G_1' \in \mathbb{R}^{r_0'\times n_1 \times r_1'}$ and $G_1'' \in \mathbb{R}^{r_0''\times n_1 \times r_1''}$ with $r_0' \geq r_0''$
\begin{align}\label{eq:addition_end11}
G_1(i_1) = \left[\begin{array}{c|c}
        &           G_1''(i_1)   \\ 
      \cline{2-2} \raisebox{.6\normalbaselineskip}{$G_1'(i_1)$}  & 0_{ \left(r_0' - r_0''\right)\times r_1''}\\
    \end{array}\right], \quad
    G_d(i_d) = \left[\begin{array}{c|c}
      \multicolumn{2}{c}{\smash{\raisebox{0\normalbaselineskip}{$G_d'(i_d)$}}} \\ \hline 
		G_d''(i_d)  & 0_{r_{d-1}''\times \left(r_0' - r_0''\right)}
    \end{array}\right].
\end{align}
If $r_0' \leq r_0''$:
\begin{align}
G_1(i_1) = \left[\begin{array}{c|c}
       G_1'(i_1) &  \\
      \cline{1-1} 0_{ \left(r_0'' - r_0'\right)\times r_1'} & \raisebox{.6\normalbaselineskip}{$G_1''(i_1)$}  \\
    \end{array}\right], \,\,
    G_d(i_d) = \left[\begin{array}{c|c}
    		G_d'(i_d)  & 0_{r_{d-1}'\times \left(r_0'' - r_0'\right)}  \\ \hline 
      \multicolumn{2}{c}{\smash{\raisebox{0\normalbaselineskip}{$G_d''(i_d)$}}}
    \end{array}\right].
    \end{align}

\end{theorem}

\begin{proof}
We consider only the case $r_0' \geq r_0''$; the remaining case is treated similarly. We have
\begin{align}
\prod_{k=1}^d G_k(i_k) 
= &\left[\begin{array}{c|c}
        &           \prod_{k=1}^{d-1}G_k''(i_k) \\
      \cline{2-2} \raisebox{.6\normalbaselineskip}{$  \prod_{k=1}^{d-1}G_k'(i_k) $}  & 0_{ \left(r_0' - r_0''\right)\times r_{d-1}''}\\
    \end{array}\right] 
      \left[\begin{array}{c|c}
      \multicolumn{2}{c}{\smash{\raisebox{0\normalbaselineskip}{$G_d'(i_d)$}}} \\ \hline 
		G_d''(i_d)  & 0_{r_{d-1}''\times \left(r_0' - r_0''\right)}
    \end{array}\right].
\end{align}
If we write the left and right matrices in the product as  $\bigl[ \begin{smallmatrix}A & C   \\ 
B & 0 \end{smallmatrix}\bigr],
\bigl[ \begin{smallmatrix} A_1 & B_1   \\ 
C_1 & 0
   \end{smallmatrix}\bigr],$ respectively, then it follows that 
      \begin{equation}
      \begin{split}
T(i_1, \ldots , i_d)
&= \text{Trace} \left(\begin{bmatrix} A & C   \\ 
B & 0
   \end{bmatrix} \begin{bmatrix} A_1 & B_1   \\ 
C_1 & 0
   \end{bmatrix} \right) = \text{Trace} \begin{bmatrix} AA_1 + CC_1 &A B_1   \\ 
BA_1 & BB_1
   \end{bmatrix} \\
   &= \text{Trace} \begin{bmatrix} AA_1 & AB_1   \\ 
BA_1 & BB_1
   \end{bmatrix} +  \text{Trace}\left( CC_1 \right) = T'(i_1, \ldots , i_d) + T''(i_1, \ldots , i_d).
   \end{split}
\end{equation}
\end{proof}

\begin{remark}\label{rem:add_choice}
Another valid choice of added cores would be to stack the cores $G'_1(i_1), G''_1(i_1)$ and the cores $G'_d(i_d), G''_d(i_d)$ as 
\begin{align}\label{eq:alt_stack_add}
\left[\begin{array}{c|c}
      \multicolumn{2}{c}{\smash{\raisebox{0\normalbaselineskip}{$G_1'(i_1)$}}} \\ \hline 
		G_1''(i_1)  & 0_{r_{0}''\times \left(r_1' - r_1''\right)}
    \end{array}\right], \qquad
\left[\begin{array}{c|c}
        &           G_d''(i_d)   \\ 
      \cline{2-2} \raisebox{.6\normalbaselineskip}{$G_d'(i_d)$}  & 0_{ \left(r_{d-1}' - r_{d-1}''\right)\times r_0''}\\
    \end{array}\right],
\end{align}
respectively, which is proved by a similar calculation. By permutation invariance, any choice of adjacent indices where the cores of one index is stacked as in the left matrix in Eq.~\eqref{eq:alt_stack_add} and the cores of the other as in the right matrix in Eq.~\eqref{eq:alt_stack_add} then defines cores of an added tensor $T = T' + T''$. In the remainder, we fix the choice in Thm.~\ref{thm:add} for the sake of definiteness.
\end{remark}

\begin{remark}\label{rem:end_cores}
Note that, unlike in previous work \cite{khoromskij2012tensors,zhao2016tensor}, we treat the end cores in a special fashion. This is required in order for the rounding procedure to reduce the size of the formally added tensor $T$ in Thm.~\ref{thm:add}. In the previous work \cite{khoromskij2012tensors,zhao2016tensor}, the formally added tensor $T$ is chosen to have all cores defined by Eq.~\eqref{eq:addition}, including $k=1, d$. In view of Prop.~\ref{prop:negrounding} below, a call of TR-rounding on $T$ would in this case have no effect at all.

\end{remark}
\begin{proposition}\label{prop:negrounding}
Let $T'$ and $T''$ be tensors with cores $G_k'$ and $G_k''$, respectively and define the tensor $T$ by its cores $G_k(i_k) = \bigl[ \begin{smallmatrix}
G_k' (i_k)& 0 \\ 0 &  G_{k}''(i_k)
\end{smallmatrix}\bigr],
$ for $k = 1, \ldots , d$ (including $k=1$ and $k=d$). If the TR-ranks of the representations of $T'$ and $T''$ are minimal, then calling Alg.~\ref{alg:TR-truncation} on $T$ with accuracy $\varepsilon = 0$ returns cores with same size as the input cores $G_k$.
\end{proposition}
\begin{proof}

The QR-sweep of Alg.~\ref{alg:TR-truncation} first performs a QR-decomposition of the matrix
\begin{equation}\label{eq:sparsitypattern}
A = \begin{bmatrix}
G_1' (1)^T& 0 & G_1'(2)^T & 0 & \ldots & 0\\ 0 &  G_{1}''(1)^T & 0 & G_{1}''(2)^T & \ldots & G_{1}''(n_1)^T
\end{bmatrix}^T.
\end{equation}
The matrices $A' = \bigl[ \begin{smallmatrix}
G_1' (1)^T& G_1'(2)^T &  \ldots & G_1'(n_1)^T
\end{smallmatrix} \bigr]^T$, and $A'' = \bigl[ \begin{smallmatrix}
G_1''(1)^T& G_1''(2)^T &  \ldots & G_1''(n_1)^T
\end{smallmatrix}\bigr]^T$ both have full column rank. To see this, assume that e.g., $A'$ could be factored as $UV$, for $U \in \mathbb{R}^{r_0n_1 \times r}$, $V\in \mathbb{R}^{r\times r_1}$, and some $r < r_1$. Then $\text{reshape}(U, [r_0, n_1, r])$, $G_2' \times_1 V, G_3', \ldots , G_d'$ would be cores of a representation of $T'$ with TR-rank vector strictly smaller than the one of $G_1', \ldots , G_d'$, which contradicts the minimality assumption.

It then follows that also $A$ has full column rank, so it has a unique reduced QR-factorization up to signs. When this is computed using the Gram-Schmidt procedure, it is clear that the decomposition $A = QR$ has the sparsity pattern
\begin{equation}
Q = \begin{bmatrix} Q_1'^T &0 & Q_2'^T &\ldots &Q_{n_1}'^T & 0 \\ 0 & Q_1''^T &0&\ldots &0& Q_{n_1}''^T \end{bmatrix}^T,
\end{equation}
and $R = \bigl[ \begin{smallmatrix} R' &0  \\
 0 & R'' \end{smallmatrix} \bigr]$. In the QR-sweep, the $1$-mode contraction of $R$ with the next tensor in the sweep, i.e., the matrix multiplication of $R$ with the unfolded matrix
 \begin{equation}
\begin{bmatrix}
G_2' (1)& 0 & G_2'(2) & 0 & \ldots & 0\\ 0 &  G_{2}''(1) & 0 & G_{2}''(2) & \ldots & G_{2}''(n_2)
\end{bmatrix}
 \end{equation}
therefore preserves the sparsity pattern of Eq.~\eqref{eq:sparsitypattern}, which continues until the last core. By a similar argument, the SVD-sweep also respects the sparsity pattern of Eq.~\eqref{eq:sparsitypattern}. This means that all matrices for which we calculate the SVD-decomposition have full rank and no ranks are reduced in the SVD-sweep. The conclusion then follows.
\end{proof}

The size of the formally added tensor $T$ in Prop.~\ref{prop:negrounding} is therefore not reduced by rounding, and storage costs in practical computations quickly explode. On the contrary, defining the addition operation as in Thm.~\ref{thm:add} leads to sizeable storage-cost reductions.

As a check, it is easy to see that an application of the rounding procedure on the tensor $T+T$ gives a result with same ranks as $T$, assuming these ranks are minimal. To see this, the first step of the QR-sweep is to perform a reduced QR-decomposition of the matrix
\begin{equation}
\begin{bmatrix}
G_1(1)^T&  G_1(2)^T & \ldots & G_{1}(n_1)^T\\ 
G_1(1)^T&  G_1(2)^T & \ldots & G_{1}(n_1)^T
\end{bmatrix}^T
= Q_1 \bigl[ R_1 R_1' \bigr],
\end{equation} 
where $Q_1$ has orthogonal columns and clearly $Q_1R_1 = Q_1R_1'$. Now, by minimality of the ranks of $T$, the matrix $\bigl[ \begin{smallmatrix} G_1(1)^T&  G_1(2)^T & \ldots & G_{1}(n_1)^T\end{smallmatrix}\bigr]^T$ has full column rank, and therefore a QR-decomposition unique up to signs. Since $Q_1R_1 = Q_1R_1'$, and both $Q_1R_1$ and $Q_1R_1'$ are QR-decompositions of this matrix, it follows that actually $R_1 = R_1'$. In the next step of the rounding procedure, the matrix $\bigl[ R_1 R_1 \bigr]$ is multiplied with 
 \begin{equation}
\begin{bmatrix}
G_2(1)& 0 & G_2(2) & 0 & \ldots & 0\\ 0 &  G_{2}(1) & 0 & G_{2}(2) & \ldots & G_{2}(n_2)
\end{bmatrix},
 \end{equation}
 which gives $\bigl[ R_1G_2(1),  R_1G_2(1) , \ldots , R_1 G_2(n_2), R_1 G_2(n_2)\bigr]$. Iterating this procedure shows that each $Q_k$ has dimensions $r_{k-1}n_k\times r_k$, for $k= 1, \ldots d$. The matrix $V_d$ is then of dimension $r_d \times r_d$ by minimality of the ranks of $r_k$, so no cyclic step is performed. By minimality, no ranks are reduced in the SVD-sweep, so the result has ranks $r_0, \ldots , r_{d-1}, r_d$.

After a draft of this manuscript appeared, Batselier \cite{batselier2018trouble} made the observation that rounding in the TR-format after applying the Hadamard product of a tensor to itself might not always reduce the resulting ranks to the original ranks. This means that different variations of Alg.~\ref{alg:TR-truncation} might need to be used for different applications. Devising and comparing these different variations remains an important open problem.

\subsection{Frobenius norm}\label{sec:frob}
In this section, we present an algorithm to compute the Frobenius norm of a tensor $T$, which improves on the cost of the algorithms previously recorded in the literature. In a previous approach \cite{zhao2016tensor}, it is suggested to form the Hadamard product $T' = T\circ T$ after which $T'$ can be contracted with a tensor of all ones. This leads to a total complexity $\mathcal{O}\left(dnr^4 + dr^6\right)$. However, the explicit formation of the large tensor $T'$ can be avoided, and the resulting Kronecker structure can be used together with permutation invariance to lower the total complexity to $\mathcal{O}\left(dnr^3 \min_k r_k\right)$.
\begin{algorithm}
\caption{Frobenius norm}\label{alg:TR-norm}
\begin{algorithmic}[1]

 \Require{$d$-tensor $T$ with cores $G_1, \ldots , G_d$.}{}
 \Ensure{$ \|T \|_F$}{} \\
 Compute $k_* = \argmin_k r_k$ and set $\tau_{k_*} = (1, d, d-1, \ldots , 2)^{k_*-1}$, $G_k'(i_k) = G_{\tau_{k_*}^{-1}(k)}(i_{\tau_{k_*}^{-1}(k)})$.\\
 $A := \sum_{i_1} G'_1(i_1)\otimes_K G'_1(i_1)$
 \For{$k = 2:d$} 
 \State $X_{\leq k}(i_k) := A\left( G'_k(i_k)\otimes_K G'_k(i_k)\right) $ using the Kronecker structure to reduce cost
 \State $A =\sum_{i_k} X_{\leq k}(i_k)$ 
 \EndFor \\
$\|T\|_F = \text{Trace}\left(A\right)$
 \end{algorithmic}
\end{algorithm}

To describe the algorithm, note that $\|T\|_F = \text{Trace}\left( X_1 \ldots X_d\right)$, where $X_k = \sum_{i_k} G_k(i_k)\otimes_K G_k(i_k)$. When computing the products $X_{\leq k} := X_1 \ldots X_k$, we can proceed from left to right and use the Kronecker structure to achieve cost $\mathcal{O}\left(knr_0r^3\right)$, so the total cost of computing $\|T\|_F = \text{Trace}\left( X_1 \ldots X_d\right)$ is $\mathcal{O}\left(dnr_0r^3\right)$. The large matrix $G_k(i_k)\otimes_K G_k(i_k)$ never has to be explicitly formed during the calculations, except for the initial matrix $G_1(i_1)\otimes_K G_1(i_1)$. Note also that $\|T\|_F$ is invariant under cyclic shifts, so we can choose the shift minimizing $r_0$, giving total cost $\mathcal{O}\left(dnr^3\min_kr_k\right)$. This is summarized in Alg.~\ref{alg:TR-norm}.

\section{Graph-based formats}\label{sec:graph}
As noted in the literature \cite{ye2018tensor}, the representation ranks and therefore the storage cost of a tensor can depend strongly on the structure of the graph used for the tensor format. It is therefore desirable to generalize the algorithms from the previous sections to general graph formats. These questions are discussed below and lead to an algorithm to, given a tensor, select a low cost graph representation, and to an alternative conversion from canonical format into TR-format, producing higher compression ratios.

\subsection{Rounding}
For a general graph, we can perform a TR-rounding procedure on each cycle in the graph and alternate between different cycles to reduce overall ranks in a procedure we now describe. Let $\mathcal{G}$ be a graph containing a cycle $C$. The graph $\mathcal{G} \smallsetminus C$ can be written in terms of its connected components $\mathcal{C}_k, k=1, \ldots , M$ as $\mathcal{G} \smallsetminus C = \bigcup_{k=1}^M \mathcal{C}_k$. The cores of each connected component determine a tensor $T_{\mathcal{C}_k}(i_{\mathcal{C}_k}, \bigtimes_{(v,w)\in e_{C \mathcal{C}_k}}\alpha_{vw}),$ where $e_{C \mathcal{C}_k}$ denotes the edges between $C$ and $\mathcal{C}_K$ and $i_{\mathcal{C}_k}$ the indices of the vertices within $\mathcal{C}_k$. Likewise, we will use $i_C$ to denote the indices of the vertices within $C$. See Fig.~\ref{fig:graph_notation} for an illustration.

\begin{figure}
\begin{center}
\begin{tikzpicture}[trim right=6cm]
    \SetVertexNormal[FillColor = gray!10]
   \Vertex[x=0 ,y=0, L=$1$]{1}
   \Vertex[x=2 ,y=0,L=$2$]{2}
   \Vertex[x=4,y=0,L=$3$]{3}
   \Vertex[x=6 ,y=0,L=$4$]{4}
   \Vertex[x=0 ,y=2,L=$5$]{5}
   \Vertex[x=2 ,y=2,L=$6$]{6}
   \Vertex[x=4 ,y=2,L=$7$]{7}
   \Vertex[x=6 ,y=2,L=$8$]{8}
   \Edge(1)(2)
   \Edge(3)(2)
      \Edge(2)(6)
      \Edge(3)(7)
      \Edge(3)(4)
      \Edge(4)(8)
      \Edge(5)(6)
      \Edge(6)(7)
      \Edge(8)(7)
              \node[draw=orange,fit=(6) (2) (3) (7) ,inner sep=2ex,thick,rounded corners,rectangle] (tmp) {};
              \node[draw=orange,fit=(1),inner sep=2ex,thick,rounded corners,rectangle] (tmp) {};
              \node[draw=orange,fit=(5),inner sep=2ex,thick,rounded corners,rectangle] (tmp) {};
              \node[draw=orange,fit=(4) (8) ,inner sep=2ex,thick,rounded corners,rectangle] (tmp) {};

\node[text width=3cm] at (4.4,-1) 
    {$C$};
    \node[text width=3cm] at (1.4,1) 
    {$\mathcal{C}_1$};
\node[text width=3cm] at (1.4,-1) 
    {$\mathcal{C}_2$};
    
    \node[text width=3cm] at (7.4,-1) 
    {$\mathcal{C}_3$};
    
\end{tikzpicture}
\caption{Illustration of the notation. This graph and cycle $C$ have $i_C = \{i_2, i_3, i_6, i_7\}, i_{\mathcal{C}_1} = i_5$, $i_{\mathcal{C}_2} = i_1$, $i_{\mathcal{C}_3} = \{i_4, i_8\}$, $e_{C\mathcal{C}_1} = (5,6)$, $e_{C\mathcal{C}_2} = (1,2)$ and $e_{C\mathcal{C}_3} = \{(3,4), (7,8)\}$.}\label{fig:graph_notation}

\end{center}
\end{figure}
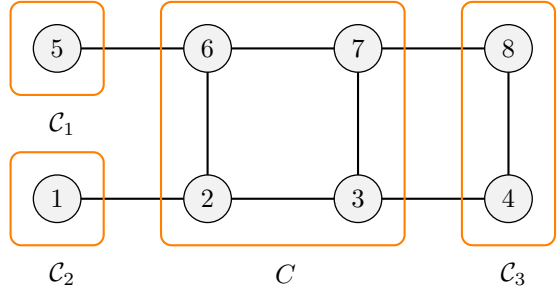

The indices $i_v$ of vertices $v$ within $\mathcal{C}_k$ incident to a vertex $w$ in $C$ can be paired up with the edge $(v,w)$, resulting in the core $T_{\mathcal{C}_k}^e(i^e_v) \in \mathbb{R}^{\bigtimes_{v\in\mathcal{C}_k} \left( \prod_{(v,w) \in e_{C \mathcal{C}_k}}n_vr_{vw}\right)},$ where for each $v \in \mathcal{C}_k$, we define $i^e_v = i_v$ for $v$ not incident to $C$, and  $i^e_v = i_v\prod_{(v,w) \in e_{C \mathcal{C}_k}}\alpha_{vw},$ written as one long index for all pairs $(v,w)$ in $e_{C \mathcal{C}_k}$. Similarly, the indices in the cycle $C$ can be paired up with the edges leading out of $C$ to define a tensor
\begin{equation}\label{eq:T_C^e}
T_{C}^e(i^e_v) \in \mathbb{R}^{\bigtimes_{v\in C} \left( \prod_{k, (v,w) \in e_{C \mathcal{C}_k}}n_vr_{vw}\right)},
\end{equation}
where for all $v \in C$, we define $i^e_v = i_v$ for $v$ not incident to any $\mathcal{C}_{k}$, and $i^e_v = i_v\prod_{k, (v,w) \in e_{C \mathcal{C}_k}}\alpha_{vw},$ written as one long index for $(v,w)$ in $e_{C \mathcal{C}_k}$. The TR-rounding procedure can be applied to $T_{C}^e(i^e_v) $, giving a resulting tensor $\widetilde{T}_C^e$, while leaving the remaining cores unaltered. The result can be reshaped back into the $\mathcal{G}$-format and we denote the result by $\widetilde{T}$. The rounding error propagates to the whole tensor as in the following, which is a straightforward generalization of a result by Handschuh \cite[Sec.~6]{handschuh2012changing}.
\begin{theorem}\label{thm:error_prop}
$\| T - \widetilde{T} \|_F \leq \| T_C^e - \widetilde{T}_C^e \|_F \prod_{j=1}^M \| T_{\mathcal{C}_j}^e \|_F.$
\end{theorem}
\begin{proof}
The triangle inequality and the fact that the Frobenius norm is a crossnorm imply
\begin{equation}
\| T - \widetilde{T} \|_F = \| \sum_{\substack{k,\alpha_{vw},\\ (v,w) \in e_{C\mathcal{C}_k}} }  \left(T_C^e(i_C^e) - \widetilde{T}_C^e(i_C^e)\right)\otimes \bigotimes_{j=1}^M T_{\mathcal{C}_j}^e(i^e_v)\|_F 
\leq  \sum_{\substack{k,\alpha_{vw},\\ (v,w) \in e_{C\mathcal{C}_k}} }   \| T_C^e(i_C^e) - \widetilde{T}_C^e(i_C^e)\|_F \prod_{j=1}^M \|T_{\mathcal{C}_j}^e(i^e_v)\|_F.
\end{equation}

Next, we use
\begin{equation}
\sum_{\beta_1 , \ldots , \beta_M} a(\beta_1, \ldots ,\beta_M) \prod_{j=1}^M b_j(\beta_j) \leq \left(\sum_{\beta_1 , \ldots , \beta_M} a(\beta_1, \ldots ,\beta_M)^2 \right)^{\frac{1}{2}} \prod_{j=1}^M\left(\sum_{\beta_j} b_j(\beta_j)^2\right)^{\frac{1}{2}},
\end{equation}
which follows by induction on $M$ from Cauchy-Schwarz, to conclude that
\begin{equation}
\| T - \widetilde{T} \|_F^2 \leq  \left( \sum_{\substack{k,\alpha_{vw},\\ (v,w) \in e_{C\mathcal{C}_k}} }  \| T_C^e(i_C^e) - \widetilde{T}_C^e(i_C^e)\|_F^2 \right) \left(\prod_{j=1}^M  \sum_{\substack{\alpha_{vw},\\ (v,w) \in e_{C\mathcal{C}_j}} }\|T_{\mathcal{C}_j}^e(i^e_v)\|_F^2  \right) = \| T_C^e - \widetilde{T}_C^e \|_F^2 \prod_{j=1}^M \| T_{\mathcal{C}_j}^e \|_F^2.
\end{equation}

\end{proof}
Because of the upper bound in Thm.~\ref{thm:error_prop}, we can achieve a rounding error $\| T - \widetilde{T} \|_F \leq \varepsilon \|T\|_F$ by performing the rounding of the cycle $C$ with rounding error 
\begin{equation}\label{eq:graph_accuracy}
 \| T_C^e - \widetilde{T}_C^e \|_F  \leq \varepsilon  \frac{\|T\|_F}{\prod_{j=1}^M \| T_{\mathcal{C}_j}^e \|_F}.
\end{equation}
Note that the second factor on the right hand side is bounded from above by $\| T_C^e\|_F$, which can be seen by an application of Thm.~\ref{thm:error_prop} with $\widetilde{T}_C^e = 0$.

We summarize the above considerations in Alg.~\ref{alg:graph_truncation}.

\begin{algorithm}
\caption{$\mathcal{G}$-truncation}\label{alg:graph_truncation}
\begin{algorithmic}[1]
 \Require{$d$-tensor $T$ with cores $G_1, \ldots , G_d$, $\mathcal{G}$-ranks $r_{ik}$, accuracy $\varepsilon$.}{}
 \Ensure{$d$-tensor $\widetilde{T}$ with cores $\widetilde{G}_k$, $\mathcal{G}$-ranks $\widetilde{r}_{ik} \leq r_{ik}$ and $\| T - \widetilde{T} \|_F \leq \varepsilon\|T\|_F$.}{} \\
 Write $\mathcal{G}$ as a union of cycles and paths.
 \For{each cycle or path $C$}
 \State Call Alg.~\ref{alg:TR-truncation} on $T_C^e$ from Eq.~\eqref{eq:T_C^e} with accuracy from Eq.~\eqref{eq:graph_accuracy}.
 \EndFor
 \end{algorithmic}
\end{algorithm}

\subsection{Transforming the graph structure}\label{sec:transforming_graph}
This section details an inexpensive way of transforming a tensor $T$ given in a $\mathcal{G}$-graph based format into a $\mathcal{G}'$-graph based format, for any two graphs $\mathcal{G}$ and $\mathcal{G}'$. 

The graph $\mathcal{G}$ can be transformed into the graph $\mathcal{G}'$ by successively inserting and deleting edges. We now describe how to capture these successive transformations on the level of the tensor $T$. For simplicity of the exposition, we describe the procedure of deleting an edge in TR-format (Alg.~\ref{alg:delete_edge}) and inserting an edge into a TT-format (Alg.~\ref{alg:insert_edge}).

\subsubsection{Edge deletion}
\begin{algorithm}
\caption{Edge deletion}\label{alg:delete_edge}
\begin{algorithmic}[1]
 \Require{TR-representation $G_1, \ldots , G_d$ of $d$-tensor $T$, edge $k$ to be deleted.}{}
 \Ensure{TR-representation of $T$ with $\widetilde{r}_{k-1} = 1$.}{} \\
Take $\tau_k = ( 1, d, d-1, \ldots, 2  )^{k-1}$
 \For{$\ell = 1:r_k$} 
 \State define $T_{\ell}$ by $T_{\ell}(i_1, \ldots , i_d) = \sum\limits_{\substack{\alpha_0, \ldots , \alpha_{d-1} \\ \alpha_d = \alpha_0}} \prod_{j = 1}^d G_{j}(\alpha_{j- 1}, i_{j}, \alpha_{j})\delta_{\alpha_k, _{\ell}}$
 \EndFor \\
 $T = \sum_{\ell = 1}^{r_k} T_{\ell}^{\tau_k}$ \\
 $ T = T^{\tau_k^{-1}}$
 \end{algorithmic}
\end{algorithm}

To describe the procedure for deleting an edge from a TR-representation, note that having $r_{k-1} = 1$ is equivalent to a TR-representation with the edge between vertices $i_{k-1}$ and $i_{k}$ deleted. By fixing all cores except for the $k$:th one, the tensor can be written as a sum of $r_{k-1}$ distinct tensors in TR-format, where each one has $k$:th core tensor $G_k(\ell, i_k, \cdot)$, for $1 \leq \ell \leq r_{k-1}$. We demonstrate in Thm.~\ref{thm:correct_delete} that this addition can be performed on the level of the cores in such a way that the result has $k-1$:th rank equal to $1$ and the full procedure is written out in Alg.~\ref{alg:delete_edge}.

\begin{theorem}\label{thm:correct_delete}
If $\widetilde{T}$ denotes the output of Alg.~\ref{alg:delete_edge} when called with input $T$, then $\widetilde{T}$ has $k-1$:th TR-rank equal to $1$ and $\widetilde{T}(i_1, \ldots , i_d) = T(i_1, \ldots , i_d)$ for all $i_1 , \ldots , i_d$.
\end{theorem}
\begin{proof}
With the notation from Alg.~\ref{alg:delete_edge}, each $T_{\ell}^{\tau_k}$ can be represented with the cores $G_k(\ell, \cdot, \cdot)$, $G_{k+1}(\cdot, \cdot, \cdot)$, $\ldots$, $G_{k-1}(\cdot, \cdot, \ell)$ for fixed $\ell$. Since the first core for each $\ell$ has size $1\times n_k \times r_k$, it follows from Thm.~\ref{thm:add} that $ \sum_{\ell = 1}^{r_k} T_{\ell}^{\tau_k}$ has first core of size $1\times \cdot \times \cdot$, from which the first claim of the theorem follows. For the second, we have
\begin{align}
\widetilde{T}(i_1, \ldots , i_d) &=  \sum_{\ell = 1}^{r_k} T_{\ell}(i_1, \ldots , i_d) =  \sum\limits_{\substack{\ell, \alpha_0, \ldots , \alpha_{d} \\ \alpha_d = \alpha_0}} \delta_{\alpha_k, _{\ell}}  \prod_{j = 1}^d G_{j}(\alpha_{j- 1}, i_{j}, \alpha_{j}) = T(i_1, \ldots , i_d).
\end{align}
\end{proof}

Alg.~\ref{alg:delete_edge} itself has negligible cost. It can (and should) however be combined with a call to the TR-rounding procedure, resulting in a total complexity of $\mathcal{O}(dnr^3r_{k-1}^3)$, since the cores of $\sum_{\ell = 1}^{r_k} T_{\ell}^{\tau_k}$ have ranks not greater than $rr_{k-1}$. Note that $\widetilde{r}_{k-1} = 1$ is unchanged by an application of TR-rounding.

\subsubsection{Edge insertion}
Next, we detail the procedure of inserting an edge between indices $i_1$ and $i_d$ in a TT-representation of $T$. This is equivalent to finding a TR-representation of $T$ with $r_0 \neq 1$ and is achieved with negligible cost when performed as in Alg.~\ref{alg:insert_edge} below. The algorithm consists of a series of operations reshaping each core tensor $G_k$ of $T$.

\begin{algorithm}
\caption{Edge insertion}\label{alg:insert_edge}
\begin{algorithmic}[1]
 \Require{TT-representation $G_1, \ldots , G_d$ of $d$-tensor $T$, divisor $\widetilde{r}_0$ of $r_1$.}{}
 \Ensure{TR-representation $\widetilde{G}_1, \ldots , \widetilde{G}_d$ of $T$ with TR-ranks $\widetilde{r}_k$.}{} 
 \State Write $G_1(i_1) = \bigl[\begin{smallmatrix}
        G_1^{(1)}(i_1) & \cdots &  G_1^{(\widetilde{r}_0)}(i_1)
    \end{smallmatrix}\bigr]$, and $
    G_2(i_2) = \bigl[\begin{smallmatrix}
        G_2^{(1)}(i_2)^T &
        \cdots &
          G_2^{(\widetilde{r}_0)}(i_2)^T
    \end{smallmatrix}\bigr]^T  ,$ where each $G_1^{(j)}(i_1) \in \mathbb{R}^{1\times \frac{r_1}{\widetilde{r}_0}}$ and $G_2^{(j)}(i_2) \in \mathbb{R}^{\frac{r_1}{\widetilde{r}_0} \times r_2}$
\State Form $\widetilde{G}_1(i_1) = \bigl[ \begin{smallmatrix}
        G_1^{(1)}(i_1)^T &
        \cdots & 
          G_1^{(\widetilde{r}_0)}(i_1)^T
    \end{smallmatrix}\bigr]^T  $  \Comment{Stack first core}
    
    \State Form $\widetilde{G}_2(i_2)  =  \bigl[\begin{smallmatrix}
        G_2^{(1)}(i_2) & \cdots &  G_2^{(\widetilde{r}_0)}(i_2)
    \end{smallmatrix}\bigr]$ \Comment{Stack second core}
  
\State Form $\widetilde{G}_k(i_k) = \text{blockdiag}\left(G_k(i_k)\right)$ for $k  = 3, \ldots , d.$  \Comment{Stack remaining cores}
\end{algorithmic}
\end{algorithm}

We now prove the correctness of Alg.~\ref{alg:insert_edge}.
\begin{theorem}
If $\widetilde{T}$ denotes the output of Alg.~\ref{alg:insert_edge} when called with input $T$, then $\widetilde{T}$ satisfies $\widetilde{T}(i_1, \ldots , i_d) = T(i_1, \ldots , i_d)$ for all $i_1 , \ldots , i_d$.
\end{theorem}
\begin{proof}
Write $G_1(i_1) = \bigl[\begin{smallmatrix}
        G_1^{(1)}(i_1) & \cdots &  G_1^{(\widetilde{r}_0)}(i_1)
    \end{smallmatrix}\bigr]$, and $
    G_2(i_2) = \bigl[\begin{smallmatrix}
        G_2^{(1)}(i_2)^T &
        \cdots &
          G_2^{(\widetilde{r}_0)}(i_2)^T
    \end{smallmatrix}\bigr]^T  ,$ where each $G_1^{(j)}(i_1) \in \mathbb{R}^{1\times \frac{r_1}{\widetilde{r}_0}}$ and $G_2^{(j)}(i_2) \in \mathbb{R}^{\frac{r_1}{\widetilde{r}_0} \times r_2}$. Alg.~\ref{alg:insert_edge} returns $
\widetilde{G}_1(i_1) = \bigl[ \begin{smallmatrix}
        G_1^{(1)}(i_1)^T &
        \cdots & 
          G_1^{(\widetilde{r}_0)}(i_1)^T
    \end{smallmatrix}\bigr]^T  ,$ $
    \widetilde{G}_2(i_2)  =  \bigl[\begin{smallmatrix}
        G_2^{(1)}(i_2) & \cdots &  G_2^{(\widetilde{r}_0)}(i_2)
    \end{smallmatrix}\bigr]$
and $\widetilde{G}_k(i_k) = \text{blockdiag}\left(G_k(i_k)\right)$ for $k  = 3, \ldots , d.$ Writing $H = \prod_{k=3}^d G_k(i_k)$, this gives
\begin{equation}
\begin{split}
\widetilde{T}(i_1 , \ldots , i_d) &= \text{Trace} \begin{bmatrix}
       G_1^{(1)}(i_1) G_2^{(1)}(i_2) &  \cdots &  G_1^{(1)}(i_1) G_2^{(\widetilde{r}_0)}(i_2) \\
       \vdots &   \ddots & \vdots \\
       G_1^{(\widetilde{r}_0)}(i_1) G_2^{(1)}(i_2) &  \cdots &  G_1^{(\widetilde{r}_0)}(i_1) G_2^{(\widetilde{r}_0)}(i_2) \\
    \end{bmatrix} \text{blockdiag}\left( H\right) \\
    &= \sum_{\alpha_1 = 1}^{\widetilde{r}_0} \text{Trace}\left( G_1^{(\alpha_1)}(i_1) G_2^{(\alpha_1)}(i_2)H\right) = \text{Trace}\left( G_1(i_k)G_2(i_k) H\right) = T(i_1, \ldots , i_d).
    \end{split}
\end{equation}
\end{proof}

Just as for Alg.~\ref{alg:delete_edge}, Alg.~\ref{alg:insert_edge} should be combined with an application of TR-rounding, resulting in a total cost $\mathcal{O}\left( dnr^3\widetilde{r}_0^3\right)$, since the TR-ranks of $\widetilde{T}$ are no greater than $r\widetilde{r}_0$.

\begin{remark}
When combined with the TR-rounding procedure, Alg.~\ref{alg:delete_edge} and Alg.~\ref{alg:insert_edge} have cost $\mathcal{O}\left( dnr^3\widetilde{r}_0^3\right)$ and $\mathcal{O}\left( dnr^3r_{k-1}^3\right)$. The question of deleting and inserting edges was considered also in \cite{handschuh2012changing}, producing algorithms with complexity $\mathcal{O}\left(dn^4r^6 + n^6r^6\right)$ and $\mathcal{O}\left(dn^4r^3\widetilde{r}_0^3\right)$, respectively. Our results therefore have complexity lower by a factor greater than $\text{max}_k n_k^3$.
\end{remark}

\begin{remark}\label{rem:faster_ins_del}
The rounding procedure applied to the result of Alg.~\ref{alg:insert_edge} can be costly, since the ranks of the resulting tensor $T$ can be large if $r_0$ is large. To mitigate this, let $\gamma = (1, d, d-1, \ldots , 2)$. Note that $T^{\gamma}$ is in the form $T^{\gamma} = \sum_{\alpha = 1}^{\widetilde{r}_0} T'_{\alpha_0}$, where $T'_{\alpha_0}$ has first core $G_2^{(\alpha_0)}(i_2)$, middle cores $G_k(i_k)$ for $k = 3, \ldots , d$ and last core $G_1^{(\alpha_0)}(i_1)$. These can be added successively, with a rounding procedure with accuracy $\frac{\varepsilon}{r_0}$ applied after each formal addition. The ranks are then reduced at every stage, leading to lower complexity. Similarly, for Alg.~\ref{alg:delete_edge}, the rounding procedure can be performed after each addition.
\end{remark}

\begin{remark}
Alg.~\ref{alg:insert_edge} can also be used to transform a TR-representation with ranks $r_k$ into a TR-representation with ranks $r_0\cdot \rho_1, \frac{r_1}{\rho_1}, r_2\cdot \rho_1 ,\ldots , r_d\cdot \rho_1$, where $\rho_1$ is any divisor of $r_1$.
\end{remark}

The framework introduced in Algs.~\ref{alg:delete_edge} - \ref{alg:insert_edge} can be applied to transform a general graph $\mathcal{G}$ into another graph $\mathcal{G}'$ by successively inserting and deleting edges. Thm.~\ref{alg:graph_truncation} can be used to reduce storage sizes of the core tensors involved in the calculations.
\subsection{Applications of edge insertion and deletion}
We now present two applications of Algs.~\ref{alg:delete_edge} - \ref{alg:insert_edge}.
\subsubsection{Greedy algorithm for selecting graph structure}
Consider a tensor $T$ for which we would like to choose a graph $\mathcal{G}$ with which to represent $T$. We introduce a cost function $f_T(\mathcal{G})$ measuring the cost associated to representing $T$ with the graph $\mathcal{G}$ (storage cost, maximum/average rank et.c.) and would like to minimize $f_T(\mathcal{G})$. Clearly, an exhaustive search over all possible graph structures is computationally intractable because of their number. We therefore limit the scope to searching over graphs $\mathcal{G}$ with a certain structure and let the set of permissible graphs be denoted by $\mathbb{G}$. We consider the problem of finding $\mathcal{G}^* = \argmin_{\mathcal{G}\in \mathbb{G}} f_T(\mathcal{G}).$ Alg.~\ref{alg:greedy_opt} details the simplest attempt at finding a low-cost graph, although we do not claim any convergence to the minimum value. Two examples attaining low cost are shown in Sec.~\ref{sec:ex_insert_edge}. 
\begin{algorithm}
\caption{Greedy algorithm for selecting graph structure}\label{alg:greedy_opt}
\begin{algorithmic}[1]
 \Require{$d$-tensor $T$, set of permissible graphs $\mathbb{G}$.}{}
 \Ensure{Graph $\mathcal{G}\in \mathbb{G}$ and $T$ in $\mathcal{G}$-graph based format.}{}\\
Take an initial $\mathcal{G} \in \mathbb{G}$. \\
Compute a representation of $T$ in $\mathcal{G}$-format using Alg.~\ref{alg:TR-SVD} and inserting or deleting edges using Algs.~\ref{alg:delete_edge} - \ref{alg:insert_edge}.
\For {$e \in \{1, \ldots , d\}^2 $} 
\State $\mathcal{G}' = \mathcal{G} \cup e$
\If {$\mathcal{G'} \in \mathbb{G}$}
 \State Compute representation of $T$ in $\mathcal{G'}$-format using Alg.~\ref{alg:insert_edge}
 \If {$f_T(\mathcal{G'}) < f_T(\mathcal{G})$}
 \State $\mathcal{G} = \mathcal{G}'$
 \EndIf
 \EndIf
\EndFor
\end{algorithmic}
\end{algorithm}

\subsubsection{Converting canonical format into TR-format}\label{sec:can_to_TR}
We also detail a conversion of a tensor in canonical format to TR-format with increased compression ratio. This can be seen as a special case of the previous section when the initial graph is the chain on $d$ elements and $\mathbb{G}$ is precisely the cycle on $d$ elements in the same order as the chain. Given a tensor $T$ in rank-$r$ canonical format $T = \sum_{i=1}^r \bigotimes_{j=1}^d v_i^{(j)},$ a TT-representation of $T$ is defined \cite{oseledets2011tensor} by
\begin{align}\label{eq:cp_TT}
\begin{split}
G_1(i_1) &= \left[ v_1 ^{(1)}(i_1), \ldots , v_r ^{(1)}(i_1) \right], \qquad G_d(i_d) = \left[ v_1 ^{(d)}(i_d), \ldots , v_r ^{(d)}(i_d) \right]^T, \\ 
G_k(i_k) &= \text{diag}\left( v_1 ^{(k)}(i_k), \ldots , v_r ^{(k)}(i_k) \right), \quad k = 2, \ldots , d-1.
\end{split}
\end{align}
In the TR-format, we now have an additional degree of freedom to choose the starting rank $r_0$ as any divisor of $r$, and therefore apply Alg.~\ref{alg:insert_edge} with this choice. The optimal choice of $r_0$ can be chosen in a loop where the representation with currently smallest storage cost is retained. A numerical example is shown in Sec.~\ref{sec:ex_CP2TR}.

\begin{remark}
Unlike in previous work \cite{khoromskij2012tensors,zhao2016tensor}, the end cores $G_1$ and $G_d$ are treated differently from the other cores. This is done in order to ensure that the rounding procedure is able to reduce the TR-ranks of the tensor (cf. Remark~\ref{rem:end_cores}).
\end{remark}

\section{Computational experiments}\label{sec:examples}
This section is devoted to numerical studies of the algorithms presented in Secs.~\ref{sec:heur} - \ref{sec:graph}. All computations were carried out on a MacBook Pro with a 3.1 GHz Intel Core i5 processor and 16 GB of memory.
\subsection{Converting from full format to TR-format}\label{sec:ex_full2TR}
We convert a tensor given in full format into a TR-representation and compute its storage cost. We compare the TT-representation with $r_0 = 1$, Alg.~\ref{alg:TR-SVD} using a balanced representation \cite{espig2012note,zhao2016tensor} with $r_0 = \argmin  \abs{r_0 - \frac{\text{rank}_{\delta}\left( T_{\langle 1 \rangle}\right)}{r_0}}$, and Alg.~\ref{alg:reduced_storage_TR-SVD}, to Alg.~\ref{alg:reducedTR-SVD}. We do not compare to other algorithms for TR-decompositions, since these have already been compared to the TR-SVD algorithm in the literature\cite{zhao2016tensor}. The tensors were taken to be discretizations of functions $f(x_1, \ldots , x_d)$ on a grid in $\mathbb{R}^d$ with $n$ discretization points in each dimension.

We first demonstrate our approach on a set of generic examples. It is easy to see that tensors with entries sampled from an absolutely continuous distribution have unfolded matrices of full rank, with probability one. These therefore lead to decompositions with high rank. We will instead study a different set of randomly generated, generic examples. We choose $f(x)$ to be a multivariate polynomial with randomly generated exponents. Specifically, we generate exponents $\alpha_{ij}$ uniformly in the set $\{ 1, \ldots , m_\text{deg}\}$, and compress the discretization of the $5$-dimensional function
\begin{equation}\label{eq:rand_poly}
\sum_{j=1}^{m_\text{term}} x_1^{\alpha_{1j}}x_2^{\alpha_{2j}}\cdot \ldots \cdot x_5^{\alpha_{5j}},
\end{equation}
into the TR-format. The parameters were set to $d = 5$, $n =20$, $\varepsilon = 10^{-12}$ and the grid to be $[0,1]^d$. The result is shown in Table~\ref{table:ex_full2TRrand} as a function of $m_\text{deg}$ and $m_\text{term}$. As an average over $100$ samples of $\alpha_{ij}$, we see that both Alg.~\ref{alg:reduced_storage_TR-SVD} and Alg.~\ref{alg:reducedTR-SVD} perform better than the TT-format, whereas Alg.~\ref{alg:TR-SVD} with the balanced representation $r_0 = \argmin  \abs{r_0 - \frac{\text{rank}_{\delta}\left( T_{\langle 1 \rangle}\right)}{r_0}}$ performs significantly worse. This shows that taking $r_0$ and the choice of cyclic shift into account can have considerable effect on the compression ratio.

\begin{table}[ht]
{\footnotesize
\caption{Storage cost and runtime in TR-format as fraction of storage cost and runtime in TT-format, for tensors in Eq.\eqref{eq:rand_poly}. Runtimes were computed averaged over $100$ random samples of $\alpha_{ij}$.}

\begin{center}
\setlength\tabcolsep{3.9pt}
\resizebox{\textwidth}{!}{\begin{tabular}{c | c c c c c | c c c c c} 
\hline \multicolumn{10}{c}{\textbf{Alg.~\ref{alg:TR-SVD}}} \\
\hline 
 
 & \multicolumn{5}{c}{Storage quotient} & \multicolumn{5}{c}{Runtime quotient} \\
\diagbox{$m_\text{deg}$}{$m_\text{term}$} &2&6&10&14&18 &2&6&10&14&18 \\
\hline
$2$  & 1  &     1 &     1  & 1 & 1 & 0.9887 & 0.9923 & 0.9974 & 0.9947 & 0.9936\\ 
$4$ & 1 & 1.4661 & 1.6734 & 1.7357 &1.6381 & 0.9950 & 1.0063 & 1.0079 & 1.0063 & 1.0067 \\
$6$& 1 & 1.6218 & 1.6359 &1.8878 & 1.9491 & 0.9924 & 1.0096 & 1.0127 & 1.0335 & 1.0440 \\
$8$& 1 & 2.0743& 1.7643 &1.7724 & 2.0492 & 0.9940 & 1.0136 & 1.0083 & 1.0159 & 1.0590 \\
$10$& 1 & 1.9221 & 1.9383 & 2.6276& 3.1629 & 0.9954 & 1.0088 & 1.0155 & 1.0837 & 1.1574 \\ [1ex] 
\hline 

\hline \multicolumn{10}{c}{\textbf{Alg.~\ref{alg:reduced_storage_TR-SVD}}} \\
\hline 
 
 & \multicolumn{5}{c}{Storage quotient} & \multicolumn{5}{c}{Runtime quotient} \\
\diagbox{$m_\text{deg}$}{$m_\text{term}$} &2&6&10&14&18 &2&6&10&14&18 \\
\hline
$2$  &     0.9093  &  0.9218  &  0.9361  &  0.9689  &  0.9783 &    12.0628  & 14.4779  & 14.6611 &  14.6384  & 14.6350 \\ 
$4$ &     0.9129  &  0.8745  &  0.8845  &  0.9114  &  0.9032 &    13.6462  & 16.3108  & 18.2493  & 18.9628  & 19.0863 \\
$6$&     0.9489  &  0.8838  &  0.8953  &  0.9131  &  0.9074  &    13.9525  & 16.8535  & 17.7344  & 20.8090  & 22.0073\\
$8$&     0.9437  &  0.9143  &  0.9006  &  0.9080  &  0.9260 &    13.9874  & 17.5574  & 19.0868  & 19.4535  & 19.8855\\
$10$&     0.9550  &  0.9287  &  0.9147  &  0.9178  &  0.9366 &    14.2610  & 17.4405  & 19.6042  & 20.6393 &  21.2999\\ [1ex] 
\hline 

\hline \multicolumn{10}{c}{\textbf{Alg.~\ref{alg:reducedTR-SVD}}} \\
\hline 
 
 & \multicolumn{5}{c}{Storage quotient} & \multicolumn{5}{c}{Runtime quotient} \\
\diagbox{$m_\text{deg}$}{$m_\text{term}$} &2&6&10&14&18 &2&6&10&14&18 \\
\hline
$2$  &     0.9412  &  0.9682  &  0.9469  &  0.9736  &  0.9783 &     5.5980  &  5.5146   & 5.5340  &  5.5154  &  5.5294\\ 
$4$ &     0.9785  &  0.9539  &  0.9293  &  0.9500  &  0.9454&     5.5566  &  5.2291  &  5.1481  &  5.0998  &  5.0395 \\
$6$&    0.9827  &  0.9360  &  0.9496  &  0.9513  &  0.9470 &     5.5174  &  5.0911  &  4.8143  &  4.6878   & 4.6005\\
$8$&     0.9812  &  0.9718  &  0.9453  &  0.9571  &  0.9764 &     5.5246  &  4.9621  &  4.6655  &  4.4752  &  4.3403\\
$10$&     0.9962  &  0.9866  &  0.9605  &  0.9588  &  0.9845 &     5.4981  &  4.9188  &  4.5702  &  4.2490  & 4.1036\\ [1ex] 
\hline 
\end{tabular}
}
\label{table:ex_full2TRrand}
\end{center}
}
\end{table}

In order to distinguish the contributions of the shift and the choice of $r_0$, we also make the following comparisons. We first determine the compression ratios when using no shift of the variables. We choose the corresponding $r_0$ by exhaustive search, the choice of $r_0$ in the heuristic Alg.~\ref{alg:reducedTR-SVD} and the balanced choice $r_0 = \argmin  \abs{r_0 - \frac{\text{rank}_{\delta}\left( T_{\langle 1 \rangle}\right)}{r_0}}$. Next, we use the shift determined by Alg.~\ref{alg:reducedTR-SVD} and find the corresponding $r_0$ by exhaustive search and the balanced choice of $r_0 = \argmin  \abs{r_0 - \frac{\text{rank}_{\delta}\left( T_{\langle 1 \rangle}\right)}{r_0}}$. The result is shown in Table~\ref{table:ex_compTRrand}. Comparing to Table~\ref{table:ex_full2TRrand}, we see that both the choice of cyclic shift and $r_0$ contribute to higher compression ratios for both Alg.~\ref{alg:reduced_storage_TR-SVD} and Alg.~\ref{alg:reducedTR-SVD}.

\begin{table}[ht]
{\footnotesize
\caption{Storage cost and runtime in TR-format as fraction of storage cost and runtime in TT-format, for tensors in Eq.\eqref{eq:rand_poly}. Runtimes were computed averaged over $100$ random samples of $\alpha_{ij}$.}

\begin{center}
\resizebox{\textwidth}{!}{\setlength\tabcolsep{3.9pt}
\begin{tabular}{c | c c c c c | c c c c c | c c c c c} 
\hline \multicolumn{10}{c}{\textbf{No cyclic shift}} \\
\hline 
 
 & \multicolumn{5}{c}{Balanced choice} & \multicolumn{5}{c}{Exhaustive search}  & \multicolumn{5}{c}{Choice in Alg.~\ref{alg:reducedTR-SVD}} \\
\diagbox{$m_\text{deg}$}{$m_\text{term}$} &2&6&10&14&18 &2&6&10&14&18&2&6&10&14&18 \\
\hline
$2$  & 1  &     1 &     1  & 1 & 1 &    0.9458  &  0.9448  &  0.9569  &  0.9786  &  0.9867 &     0.9512  &  0.9782  &  0.9569   & 0.9836  &  0.9883\\ 
$4$ & 1 & 1.4661 & 1.6734 & 1.7357 &1.6381 & 0.9654  &  0.9350  &  0.9303  &  0.9438  &  0.9437 &     0.9885  &  0.9639   & 0.9393  &  0.9600  &  0.9554
\\
$6$& 1 & 1.6218 & 1.6359 &1.8878 & 1.9491 & 0.9752  &  0.9410  &  0.9517  &  0.9544  &  0.9480 &     0.9927  &  0.9460  &  0.9596  &  0.9613  &  0.9570
 \\
$8$& 1 & 2.0743& 1.7643 &1.7724 & 2.0492 &  0.9775  &  0.9569  &  0.9510  &  0.9494  &  0.9692 &     0.9912  &  0.9818   & 0.9553  &  0.9671  &  0.9864
\\
$10$& 1 & 1.9221 & 1.9383 & 2.6276& 3.1629 & 0.9812  &  0.9630  &  0.9590  &  0.9629  &  0.9747 &    1.0063  &  0.9966   & 0.9705  &  0.9688  &  0.9945 \\ [1ex] 
\hline 

\hline \multicolumn{10}{c}{\textbf{Cyclic shift as in Alg.~\ref{alg:reduced_storage_TR-SVD}}} \\
\hline

 & \multicolumn{5}{c}{Balanced choice} & \multicolumn{5}{c}{Exhaustive search}  & \multicolumn{5}{c}{Choice in Alg.~\ref{alg:reducedTR-SVD}} \\
\diagbox{$m_\text{deg}$}{$m_\text{term}$} &2&6&10&14&18 &2&6&10&14&18&2&6&10&14&18 \\
\hline
$2$  &         0.9412  &  0.9682  &  0.9469  &  0.9736  &  0.9783 &
         0.9299  &  0.9360  &  0.9419  &  0.9703  &  0.9783 &
    0.9412  &  0.9682  &  0.9469  &  0.9736  &  0.9783
\\ 
$4$ &         0.9785  &  1.3497  &  1.7010   & 1.7979  &  1.6809 &
         0.9579  &  0.9330  &  0.9138 &   0.9393  &  0.9387 &
    0.9785  &  0.9539 &   0.9293  &  0.9500  &  0.9454
\\
$6$&        0.9827  &  1.6332  &  1.6542  &  1.7661  &  1.8563 &
         0.9714 &   0.9247  &  0.9385  &  0.9446  &  0.9355 &
    0.9827 &   0.9360  &  0.9496  &  0.9513  &  0.9470
\\
$8$&         0.9812  &  2.0878  &  1.7566  &  1.8493  &  1.9250 &
         0.9663  & 0.9494  &  0.9361 &   0.9440  &  0.9653 &
    0.9812  &  0.9718  &  0.9453  &  0.9571  &  0.9764
\\
$10$&         0.9962  &  1.7905  &  2.0889  &  2.4774  &  3.0893 &
         0.9812  &  0.9643  &  0.9453  &  0.9495  &  0.9734 &
    0.9962  &  0.9866  &  0.9605  &  0.9588  &  0.9845
\\ [1ex] 
\hline 
\end{tabular}
}
\label{table:ex_compTRrand}
\end{center}
}
\end{table}

Lastly, we study a few examples in greater detail, shown in Table~\ref{table:ex_full2TR}. Unless otherwise noted, the parameters were set to $d = 5$, $n =20$, $\varepsilon = 10^{-12}$ and the grid to be $[0,1]^d$. The balanced choice of $r_0$ is ambiguous in that both $r_0$ and $\frac{\text{rank}_{\delta}\left( T_{\langle 1 \rangle}\right)}{r_0}$ minimize the expression and we report the choice with the highest resulting storage cost.

The first function in Table~\ref{table:ex_full2TR} has a coupling of the $x_1$ and $x_d$ variables. Because of Ex.~\ref{ex:why_shift}, one could therefore expect a correctly chosen cyclic shift of the variables to lead to significant storage savings, which is indeed the case. The second example has couplings between both $x_1$, $x_d$, and $x_1$, $x_2$. One might then expect the balanced choice of $r_0$ in Alg.~\ref{alg:TR-SVD} to lead to good compression ratios. It is therefore somewhat surprising that this is not the case, and that Alg.~\ref{alg:reduced_storage_TR-SVD} gives an order of magnitude larger compression ratio.

The results suggest again that finding the optimal choice of $r_0$ and cyclic shift $\tau$ are crucial for storage savings when using the TR-format. Finally, we study if multiple cyclic shifts can attain the optimal compression ratios. For each choice of cyclic shift $(1, d, d-1, \ldots, 2)^k$, for $k=0, \ldots , d-1$, we choose the corresponding $r_0$ both by exhaustive search and the balanced choice of $r_0 = \argmin  \abs{r_0 - \frac{\text{rank}_{\delta}\left( T_{\langle 1 \rangle}\right)}{r_0}}$. The results are shown in Table~\ref{table:ex_full2TRdetail}, and indicate that several shifts can lead to the highest compression ratio.

\begin{table}[ht]
{\footnotesize
\caption{Storage cost and runtime in TR-format as fraction of storage cost and runtime in TT-format, for tensors in Sec.~\ref{sec:ex_full2TR}. Runtimes were computed averaged over $100$ function calls. Alg.~\ref{alg:TR-SVD}$^*$ is the version of Alg.~\ref{alg:TR-SVD} in which $r_0$ is computed by $r_0 = \argmin  \abs{r_0 - \frac{\text{rank}_{\delta}\left( T_{\langle 1 \rangle}\right)}{r_0}}$.}

\begin{center}
\setlength\tabcolsep{3.9pt}
\begin{tabular}{c | c c c | c c c} 
\hline\hline 
 &\multicolumn{3}{c}{Storage quotient} & \multicolumn{3}{c}{Runtime quotient} \\
$f(x)$& Alg.~\ref{alg:TR-SVD}$^*$ & Alg.~\ref{alg:reduced_storage_TR-SVD}& Alg.~\ref{alg:reducedTR-SVD}& Alg.~\ref{alg:TR-SVD}$^*$ & Alg.~\ref{alg:reduced_storage_TR-SVD}& Alg.~\ref{alg:reducedTR-SVD} \\[0.5ex]
\hline
$\exp{\left( \cos(x_1x_d + \sum_{k=2}^{d-1} x_k)\right)}$  & 0.518  &     0.070 &     0.070  & 1.694 & 19.168 & 2.431\\ 
$\exp{\left( \cos(x_1x_d + x_1x_2 + \sum_{k=3}^{d-1} x_k)\right)}$ & 1.796 & 0.298 & 0.298 & 1.171 &25.218 &2.629 \\
Park function 1 \cite{simulationlib}& 0.941 & 0.217 & 0.217 &0.661 & 15.158& 1.563\\
($d=4$, grid $[10^{-10}, 1]^d$) & & & & & & \\
$\left(1 + \sum_{k=1}^d x_k^2\right)^{-\frac{1}{2}}$ & 2.776 & 1 & 1 &1.153 & 24.663 & 3.407\\
 $\exp\left( \sum_{k=1}^3 x_k x_{k+1}x_{k+2} + x_4 x_5  x_1\right)$& 1.2771 & 0.7674 & 1 &1.1257 &21.5445 &3.1206 \\ [1ex] 
\hline 
\end{tabular}
\label{table:ex_full2TR}
\end{center}
}
\end{table}

\begin{table}[ht]
{\footnotesize
\caption{Storage cost in TR-format as fraction of storage cost in TT-format and as function of the cyclic shift $(1, d, d-1, \ldots, 2)^k$. $r_0$ is chosen by exhaustive search and the balanced choice of $r_0$ in Alg.~\ref{alg:TR-SVD}.}

\begin{center}
\setlength\tabcolsep{3.9pt}
\resizebox{\textwidth}{!}{\begin{tabular}{c | c c c c c | c c c c c} 
\hline\hline 
 &\multicolumn{5}{c}{Exhaustive search} & \multicolumn{5}{c}{Balanced choice} \\ 
$k$& $0$ & $1$ & $2$ & $3$ & $4$ & $0$ & $1$ & $2$ & $3$ & $4$  \\[0.5ex]
\hline
 $\exp\left( \sum_{k=1}^3 x_k x_{k+1}x_{k+2} + x_4 x_5  x_1\right)$ &1 & 1 & \textbf{0.7674} & \textbf{0.7674} & 1& 1.2771 & 1.8851 & 1.7819 & 1.3999 & 1.8851 \\ 
$\exp{\left( \cos(x_1x_d + x_1x_2 + \sum_{k=3}^{d-1} x_k)\right)}$ & 1 &\textbf{0.2982}&\textbf{0.2982}&\textbf{0.2982}&\textbf{0.2982}& 1.796 & 0.7115 & \textbf{0.2982} & 1.0406 & 0.5555 \\ 
Park function 1 \cite{simulationlib}  & \textbf{0.2169} & \textbf{0.2169} & \textbf{0.2169}& \textbf{0.2169} & - & 0.941 &0.8434 & 0.9478 & 0.9442 & -\\ [1ex] 
\hline 
\end{tabular}
}
\label{table:ex_full2TRdetail}
\end{center}
}
\end{table}

We finally include an example of compression of a real-world dataset. The Coil-100 dataset \cite{coil100} consists of photographs of $100$ objects at $70$ different viewing angles. We subsample each of these images to an $n\times n$ image, for varying $n$. This results in a tensor in $\mathbb{R}^{100\times 70 \times 3 \times n \times n}$, which we reshape into a tensor in $\mathbb{R}^{10\times 10\times 10 \times 7 \times 3 \times n \times n}$. The storage and run-times resulting from the compression algorithms are shown in Table~\ref{table:ex_coil2TR}. Note that the heuristic Alg.~\ref{alg:reducedTR-SVD} results in a near-optimal storage cost with very low run-time.

\begin{table}[ht]
{\footnotesize
\caption{Storage cost and runtime in TR-format as fraction of storage cost and runtime in TT-format, for the Coil-100 dataset. Alg.~\ref{alg:TR-SVD}$^*$ is the version of Alg.~\ref{alg:TR-SVD} in which $r_0$ is computed by $r_0 = \argmin  \abs{r_0 - \frac{\text{rank}_{\delta}\left( T_{\langle 1 \rangle}\right)}{r_0}}$.}

\begin{center}
\setlength\tabcolsep{3.9pt}
\begin{tabular}{c | c c c | c c c} 
\hline\hline 
 &\multicolumn{3}{c}{Storage quotient} & \multicolumn{3}{c}{Runtime quotient} \\
$n$& Alg.~\ref{alg:TR-SVD}$^*$ & Alg.~\ref{alg:reduced_storage_TR-SVD}& Alg.~\ref{alg:reducedTR-SVD}& Alg.~\ref{alg:TR-SVD}$^*$ & Alg.~\ref{alg:reduced_storage_TR-SVD}& Alg.~\ref{alg:reducedTR-SVD} \\[0.5ex]
\hline
$24$  & 1.617  & 0.750 & 0.783 & 1.743 & 30.62 & 0.943\\ 
$32$ & 2.030 & 0.650 & 0.655 & 1.322 &22.35 &0.732 \\
$40$& 2.384 & 0.540 & 0.540 &0.759 & 15.16& 0.352 \\ [1ex] 
\hline 
\end{tabular}
\label{table:ex_coil2TR}
\end{center}
}
\end{table}

\subsection{Implicit PDE-solvers}\label{sec:impl-pde}
This section considers simple implicit finite difference solvers of linear PDEs and compares the storage cost and runtime of solvers in the TR- and TT-format, respectively. We consider the wave equation with zero Dirichlet boundary conditions
\begin{align}\label{eq:wave}
&\begin{cases}
\frac{\partial^2 u}{\partial t^2}(x,t) &= \Delta u(x,t), \quad x \in [0,1]^d, t\in [0, T] \\
u(x,0) &= u_0(x), \\
u(x,t) &= 0 \text{ on the boundary of } [0,1]^d.
\end{cases}
\end{align}
For the simple solver, we firstly represent a first-order finite difference approximation of the Laplacian in TR- and TT-format. We consider a $u_0$ given in canonical format with $n=2^5$ discretization points in each dimension and convert it into TT- and into the optimal TR-format in Sec.~\ref{sec:can_to_TR}, using Rem.~\ref{rem:faster_ins_del}. Afterwards, implicit Euler is used as time-marching, and we explicitly compute the inverse $\left(I - \left(\Delta t\right)^2\Delta\right)^{-1}$ to be applied at each time step in the TT-format, using the procedure of Oseledets and Dolgov \cite{oseledets2012solution}. After each iteration, we perform a rounding using either TT-rounding, or TR-rounding. We choose $u_0(x)$ to have significantly higher compression ratio in the TR-format than the TT-format and investigate whether or not the TR-rounding procedure manages to keep the compression ratio high, even after a large number of iterations. We show the results for the two initial functions
\begin{align}
u_{0,1}(x) = \sum_{k=1}^{20}\sin(kx_1)\sin(kx_d), \qquad
u_{0,2}(x) = x_1x_d + \sum_{k=2}^{d-1}x_k,
\end{align}
in Table~\ref{table:ex_wave}. The parameters used were $\varepsilon = 10^{-12}$, $\Delta t = 10^{-3}$, and $n=2^5$ discretization points in each dimension. The results show that the TR-format maintains lower storage cost than the TT-format after many iterations, possibly at the expense of longer runtime.

\setlength\tabcolsep{6pt}
\begin{table}[ht]
{\footnotesize
\caption{Storage cost in TR-format as fraction of storage cost in TT-format after $500$ time steps, for tensors in Sec.~\ref{sec:impl-pde}.} 
\begin{center}
\begin{tabular}{c | c c | c c} 
\hline\hline 

& \multicolumn{2}{c}{$u_{0,1}(x)$}  & \multicolumn{2}{c}{$u_{0,2}(x)$}\\
$d$& Storage quotient & Runtime quotient & Storage quotient & Runtime quotient \\[0.5ex]
\hline
$5$  &  0.0499 &   1.2195    &   0.5322  & 1.3599\\ 
$10$ & 0.0225 & 0.5021 & 0.4840 & 1.0507\\
$20$  & 0.0170 & 0.3632 & 0.4778 & 0.9796\\
 $30$ & 0.0295 & 0.9246& 0.5212&1.0240\\ [1ex] 
\hline 
\end{tabular}
\label{table:ex_wave} 
\end{center}
}
\end{table}

\subsection{Converting canonical format to TR-format}\label{sec:ex_CP2TR}
We consider the discretization of a function given in canonical format on a grid in $[0,1]^d$ with $n = 2^5$ discretization points in each dimension. Converting the resulting canonical decomposition into the TR-format in Eq.~\eqref{eq:cp_TT} and using Rem.~\ref{rem:faster_ins_del}, we find the optimal storage cost among all permutations and choices of $r_0$ and compare the costs to the TT-format obtained by $r_0 = 1$. The precision of the rounding was set to $\varepsilon = 10^{-12}$. We consider the four examples

\begin{align}\label{eq:ex_CP2TR_tensors}
T_1 &= x_1x_d + \sum_{k=2}^{d-1}x_k, \qquad\qquad\qquad \qquad T_2 = \sum_{k=1}^{\frac{d}{2}} x_k x_{d-k+1}, \quad d \text{ even} \\
T_3 &=  \sum_{k=1}^{\frac{d}{2}-1} x_k x_{k+1}x_{d-k}x_{d+1-k}, \quad d \text{ even}, \quad\qquad T_4 = \sum_{j=1}^{10} x_1^{\alpha_{1j}}x_2^{\alpha_{2j}}\cdot \ldots \cdot x_d^{\alpha_{dj}},
\end{align}
where the $\alpha_{ij}$ are positive integers generated uniformly from the set $\{1, \ldots , 20\}$, and averaged over $100$ samples. Since the variables $x_1$,$x_d$ in $T_1$ are coupled, one can expect high compression ratios after choosing an appropriate cyclic shift, but it is less obvious if this is also true for the tensors $T_2,T_3$ and $T_4$. Fig.~\ref{fig:ex_CP2TR} shows the results, with good storage savings for a range of dimensions and functions.

\begin{figure}
\begin{center}
  \includegraphics[width=\textwidth]{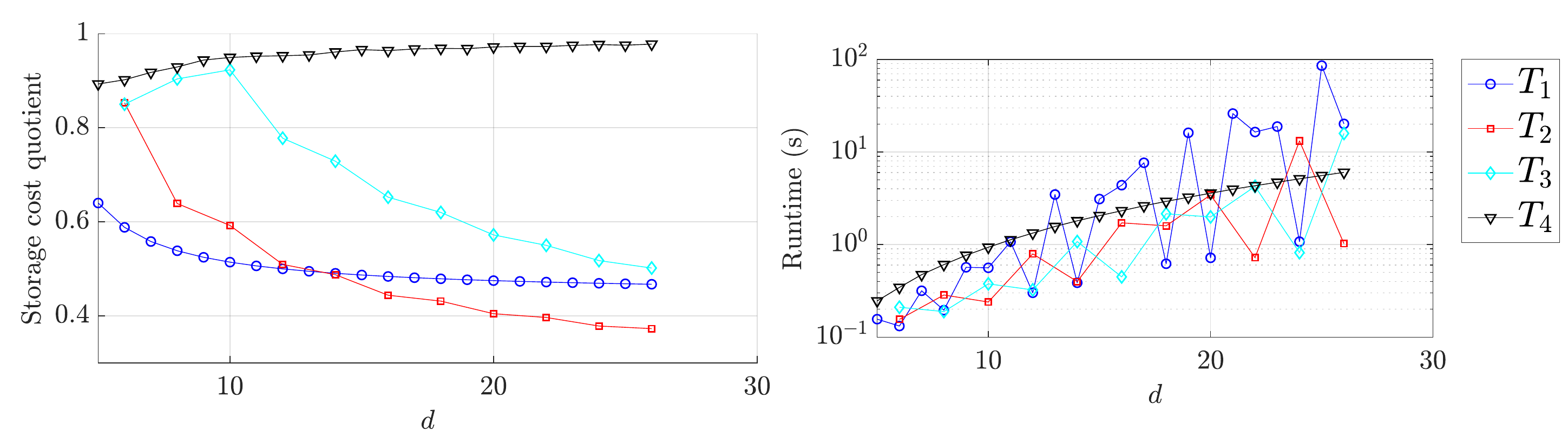}
  \caption{Left: Storage cost of TR-representation of tensors in Sec.~\ref{sec:ex_CP2TR} as fraction of storage cost in TT-format and as function of $d$. Right: runtime in seconds.}
  \label{fig:ex_CP2TR}
  \end{center}
  \end{figure}

\subsection{Selecting graph structure}\label{sec:ex_insert_edge}
We apply Alg.~\ref{alg:greedy_opt} to discretizations of the two functions
\begin{align}
f(x) = \sum_{k=1}^{20}\sin(kx_1)\sin(kx_{\lfloor \frac{d}{2} \rfloor}), \quad g(x) = \sum_{k=1}^{\frac{d-1}{2}} x_{2k-1} x_{2k+1}, \quad d \text{ odd},
\end{align}
given in the canonical format. The grid $[0,1]^d$ used $n=2^5$ grid points in each dimension. The set of permissible graph structures $\mathbb{G}$ for $f$ and $g$ were set to be all cycles with vertices $i_1, \ldots , i_d$ and exactly zero or one chords, and all chains with edges between every other index $i_{k}, i_{k+2}$, for $k = 1, 3, 5, \ldots , d-2$, respectively. For $f$, we repeated the application of Alg.~\ref{alg:greedy_opt} to compare all graphs with exactly one chord. The initial graph format was the TR-format and the TT-format, respectively. The rounding accuracy was set to $\varepsilon = 10^{-9}$. The result is presented in Fig.~\ref{fig:insert_edge} and clearly shows storage savings compared to the TT- and TR-formats.

\begin{figure}
\begin{center}
  \includegraphics[width=0.5\textwidth]{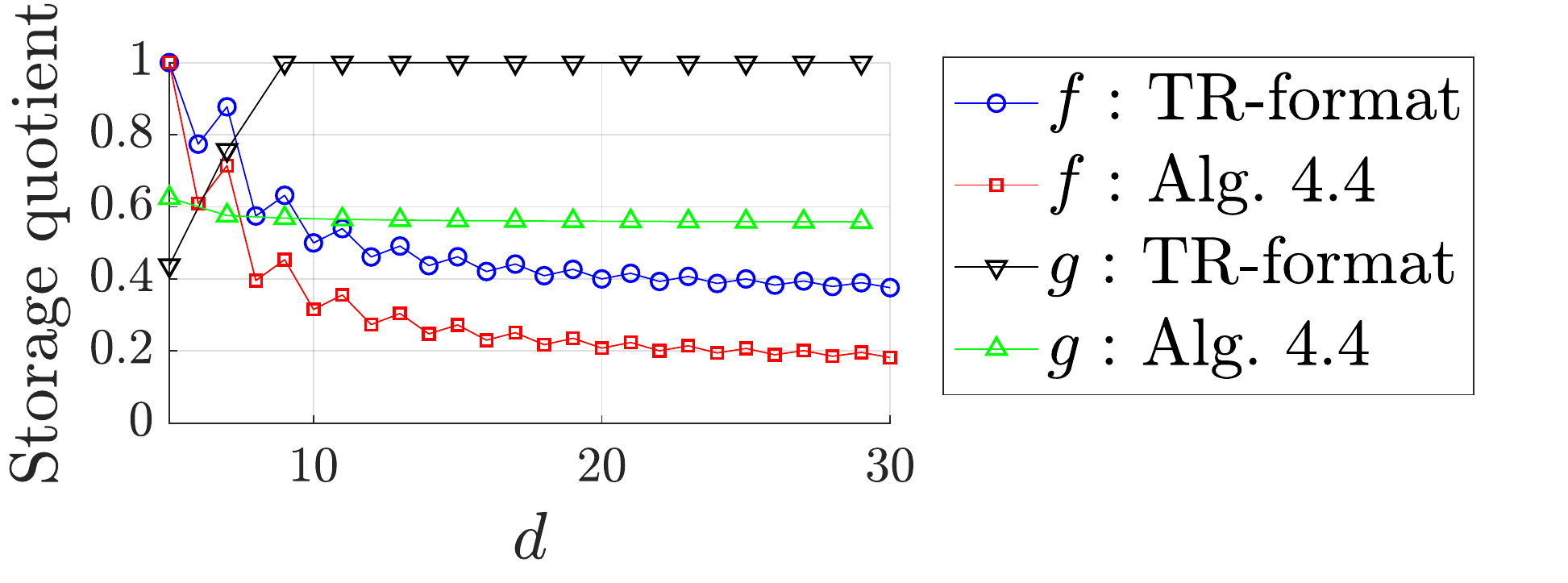}
  \caption{Storage cost of representations of tensors in Sec.~\ref{sec:ex_insert_edge} as fraction of storage cost in TT-format and as function of $d$.}
  \label{fig:insert_edge}
  \end{center}
\end{figure}

\section{Conclusions}
In this paper, we studied efficient tensor representations and computation in the tensor ring format. Firstly, we showed theoretically and numerically how two degrees of freedom in SVD-based algorithms for converting a tensor in full format into TR-format are crucial for obtaining low-cost representations. A heuristic algorithm achieving a low-cost representation with low runtime was introduced and tested numerically. Secondly, we presented a rounding procedure for the tensor ring format and showed how this required redefining common linear algebra operations to obtain a reduction of the storage-cost. Lastly, we devised algorithms for transforming the graph structure of a tensor in a graph-based format, producing even higher compression ratios. Numerical examples achieved up to more than an order of magnitude higher compression ratios than previous approaches to using the tensor ring format, without significantly affecting the runtime. An important direction for future work would be to also extend our algorithms to the case of incomplete tensors, where not all entries are used for the decomposition into the tensor-ring format \cite{khoo2017efficient}. Another important direction is to obtain a more thorough understanding of the geometric and algebraic properties of the TR-format.

\bibliographystyle{wileyNJD-AMS}
\bibliography{references}
\end{document}